\documentclass[12pt,reqno]{amsart}

\usepackage{tikz}
\usepackage{float}

\usepackage{multirow}
\usepackage{hhline}

\usepackage{mathtools}
\usepackage{times}
\usepackage[T1]{fontenc}
\usepackage{mathrsfs}
\usepackage{latexsym}
\usepackage[titletoc, title]{appendix}
\usepackage{amsmath,amsfonts,amsthm,amssymb,amscd}
\usepackage{hyperref}
\usepackage{amsmath}
\usepackage[utf8]{inputenc}

\usepackage{color}
\usepackage{breakurl}

\usepackage{comment}
\newcommand{\bburl}[1]{\textcolor{blue}{\url{#1}}}

\newcommand{\Max}{\text{\rm MAX}}

\newtheorem{thm}{Theorem}[section]

\newtheorem{lem}[thm]{Lemma}
\newtheorem{prop}[thm]{Proposition}
\newtheorem{exa}[thm]{Example}
\newtheorem{que}[thm]{Question}
\newtheorem{defi}[thm]{Definition}
\newtheorem{rek}[thm]{Remark}

\usepackage[utf8]{inputenc}

\DeclareFixedFont{\ttb}{T1}{txtt}{bx}{n}{12} 
\DeclareFixedFont{\ttm}{T1}{txtt}{m}{n}{12}

\usepackage{color}
\definecolor{deepblue}{rgb}{0,0,0.5}
\definecolor{deepred}{rgb}{0.6,0,0}
\definecolor{deepgreen}{rgb}{0,0.5,0}

\usepackage{listings}

\newcommand\pythonstyle{\lstset{
language=Python,
basicstyle=\ttm,
morekeywords={self},              
keywordstyle=\ttb\color{deepblue},
emph={MyClass,__init__},         
emphstyle=\ttb\color{deepred},    
stringstyle=\color{deepgreen},
frame=tb,                         
showstringspaces=false
}}

\lstnewenvironment{python}[1][]
{
\pythonstyle
\lstset{#1}
}
{}

\newcommand\pythoninline[1]{{\pythonstyle\lstinline!#1!}}

\DeclareMathOperator{\supp}{supp}

\numberwithin{equation}{section}

\DeclareFontFamily{U}{mathx}{}
\DeclareFontShape{U}{mathx}{m}{n}{<-> mathx10}{}
\DeclareSymbolFont{mathx}{U}{mathx}{m}{n}
\DeclareMathAccent{\widehat}{0}{mathx}{"70}
\DeclareMathAccent{\widecheck}{0}{mathx}{"71}

\makeatletter
\def\Ddots{\mathinner{\mkern1mu\raise\p@
\vbox{\kern7\p@\hbox{.}}\mkern2mu
\raise4\p@\hbox{.}\mkern2mu\raise7\p@\hbox{.}\mkern1mu}}
\makeatother

\begin{document}

\title{Lower-Order Refinements of Greedy Approximation}

\author{Kevin Beanland}
\address{Department of Mathematics, Washington and Lee University, VA 24450, USA}
\email{beanlandk@wlu.edu}

\author{H\`ung Vi\d{\^e}t Chu}
\address{Department of Mathematics, Washington and Lee University, VA 24450, USA}
\email{hchu@wlu.edu}

\author{Thomas Schlumprecht}
\address{Department of Mathematics, Texas A\&M University, TX 77843, USA and Faculty of Electrical Engineering, Czech Technical University in Prague, Zikova 4, 16627, Prague, Czech Republic}
\email{t-schlumprecht@tamu.edu}

\author{Andr\'{a}s Zs\'{a}k}
\address{Peterhouse, Cambridge CB2 1RD and Department of Pure Mathematics and Mathematical Statistics,
Centre for Mathematical Sciences, University of Cambridge, Wilberforce Road, Cambridge CB3 0WB,
United Kingdom}
\email{a.zsak@dpmms.cam.ac.uk}



\thanks{}

\maketitle

\begin{abstract}
For two countable ordinals $\alpha$ and $\beta$, a basis of a Banach space $X$ is said to be $(\alpha, \beta)$-quasi-greedy if it is
\begin{enumerate}
    \item quasi-greedy,
    \item $\mathcal{S}_\alpha$-unconditional but not $\mathcal{S}_{\alpha+1}$-unconditional, and
    \item $\mathcal{S}_\beta$-democratic but not $\mathcal{S}_{\beta+1}$-democratic.
\end{enumerate}
If $\alpha$ or $\beta$ is replaced with $\infty$, then the basis is required to be unconditonal or democratic, respectively. Previous work constructed a $(0,0)$-quasi-greedy basis, an $(\alpha, \infty)$-quasi-greedy basis, and an $(\infty, \alpha)$-quasi-greedy basis. In this paper, we construct $(\alpha, \beta)$-quasi-greedy bases for $\beta\le \alpha+1$ (except the already solved case $\alpha = \beta = 0$).
\end{abstract}

\tableofcontents

\section{Introduction}
Let $X$ be a separable Banach space over the field $\mathbb{F} = \mathbb{R}$ or $\mathbb{C}$ and $X^*$ be its dual. A countable collection $(e_i)_{i=1}^\infty \subset X$ is called a \textit{(semi-normalized) Schauder basis} if $0 < \inf_i \|e_i\|\le \sup_i \|e_i\| <\infty$, and for each $x\in X$, there is a unique sequence of scalars $(a_i)_{i=1}^\infty$ such that 
$x = \sum_{i=1}^\infty a_i e_i$. In fact, if $(e_i^*)_{i=1}^\infty\subset X^*$ is the unique sequence satisfying
$$e_i^*(e_j) \ =\ \begin{cases} 1\mbox{ if }i = j,\\ 0\mbox{ otherwise},\end{cases}$$ 
then $a_i = e_i^*(x)$ for all $i\ge 1$. Thus, $e_i^*(x)$ is also called the $i$\textsuperscript{th} coefficient of $x$. Konyagin and Temlyakov \cite{KT99} studied the greedy approximation method that kept the absolutely largest coefficients of the vector to be approximated. There they defined for a vector $x$ in a Banach space with a basis a greedy set of order $m\in \mathbb{N}$, denoted by $\Lambda(x, m)$, to contain the $m$ largest coefficients (in modulus) of $x$, i.e.,
$$\min_{i\in \Lambda(x, m)} |e_i^*(x)|\ \ge\ \max_{i\notin \Lambda(x,m)}|e_i^*(x)|.$$
An \textit{m\textsuperscript{th} greedy approximation of $x$} is the finite sum
$$\mathcal{G}_m(x)\ :=\ \sum_{i\in \Lambda(x, m)} e_i^*(x)e_i.$$
For general Banach spaces $X$ and vectors $x$, it is not necessary that $\lim_{m\rightarrow\infty}\mathcal{G}_m(x) = x$; when the convergence occurs for all $x$, the corresponding basis is said to be \textit{quasi-greedy}. Equivalently (\cite[Theorem 1]{Woj}), there is $C>0$ so that 
$$\|\mathcal{G}_m(x)\|\le C\|x\|, \mbox{ for all } x\in X \mbox{ and } m\in \mathbb{N}.$$

To measure how well $\mathcal{G}_m(x)$ approximates $x$, Konyagin and Temlyakov compared the error $\|x-\mathcal{G}_m(x)\|$ with the smallest error resulting from an arbitrary $m$-term linear combination. They called a basis \textit{greedy} if there is a constant $C > 0$ such that 
\begin{equation*}
    \|x-\mathcal{G}_m(x)\|\ \le\ C\inf_{|A|\le m, (a_i)_{i\in A}\subset\mathbb{F}}\left\|x-\sum_{i\in A} a_ie_i\right\|,\mbox{ for all } x\in X\mbox{ and } m\in \mathbb{N}.
\end{equation*}
In this case, $\mathcal{G}_m(x)$ is essentially the best $m$-term approximation of $x$ (up to the constant $C$). Greedy bases are characterized by unconditionality and democracy. Here a basis is \textit{unconditional} if there is a constant $C>0$ such that for all scalars $(a_i)_{i=1}^N$ and $(b_i)_{i=1}^N$ with $|a_i|\le |b_i|$, we have
$$\left\|\sum_{i=1}^N a_ie_i\right\|\ \le\ C\left\|\sum_{i=1}^N b_ie_i\right\|.$$
On the other hand, a basis is \textit{democratic} if for some $C > 0$, 
$$\left\|\sum_{i\in A} e_i\right\|\ \le\ C\left\|\sum_{i\in B}e_i\right\|, \mbox{for all finite }A,B \subset \mathbb{N} \mbox{ with }|A|\le |B|.$$
We use $[\mathbb{N}]^{<\infty}$ to denote the collection of finite subsets of $\mathbb{N}$ and use $1_A$ for $\sum_{i\in A}e_i$, given $A\in [\mathbb{N}]^{<\infty}$. Both unconditionality and democracy are strong properties, rendering greedy bases often nonexistent in direct sums of distinct spaces such as $\ell_p\oplus \ell_q$ $(1 \le p < q < \infty)$ and several Besov spaces \cite{DFOS}. 

Dilworth et al.\ \cite{DKKT03} made the first attempt to weaken the greedy condition while ensuring the new notion of bases has a desirable approximation capacity. They defined \textit{almost greedy} bases, for which, there exists $C> 0$ such that 
\begin{equation*}
    \|x-\mathcal{G}_m(x)\|\ \le\ C\inf_{|A|\le m}\left\|x-P_A(x)\right\|, \mbox{ for all } x\in X \mbox{ and } m\in \mathbb{N},
\end{equation*}
where $P_A(x):= \sum_{i\in A}e_i^*(x)e_i$. For almost greedy bases, the $m$-term greedy approximation $\mathcal{G}_m(x)$ is essentially the best projection in approximating $x$. It turned out that a basis is almost greedy if and only if it is quasi-greedy and democratic.

With the same goal of weakening the greedy condition, for each countable ordinal $\alpha$, the first two named authors \cite{BC} introduced and characterized $\mathcal{S}_{\alpha}$-greedy bases, where $\mathcal{S}_\alpha$ is the Schreier family of order $\alpha$. We shall define Schreier families and record their properties in Section \ref{Schreierfam}. There we see that Schreier families $\mathcal{S}_{\alpha}$ form a rich subcollection of $[\mathbb{N}]^{<\infty}$ and are essentially well-ordered by inclusion. 
These properties make the Schreier families an excellent tool for classifying bases into  various levels of approximation capacities. 

\begin{defi}\normalfont
For each countable ordinal $\alpha$, a basis is said to be \textit{$\mathcal{S}_\alpha$-greedy} if there is $C > 0$ such that 
\begin{equation*}
    \|x-\mathcal{G}_m(x)\|\ \le\ C\inf_{\substack{A\in \mathcal{S}_\alpha, |A|\le m,\\ (a_i)_{i\in A}\subset\mathbb{F}}}\left\|x-\sum_{i\in A} a_ie_i\right\|, \mbox{ for all } x\in X\mbox{ and } m\in \mathbb{N}.
\end{equation*}
\end{defi}

To characterize $\mathcal{S}_\alpha$-greedy bases, we need the notion of $\mathcal{S}_\alpha$-unconditional and $\mathcal{S}_\alpha$-democratic bases. A basis is \textit{$\mathcal{S}_\alpha$-unconditional} if for some $C > 0$,
$$\|P_A(x)\|\ \le\ C\|x\|, \mbox{ for all } x\in X \mbox{ and } A\in \mathcal{S}_\alpha.$$
A basis is \textit{$\mathcal{S}_\alpha$-democratic} if for some $C > 0$, 
$$\|1_A\|\ \le\ C\|1_B\|, \mbox{ for all } A\in \mathcal{S}_\alpha\mbox{ and } B\in [\mathbb{N}]^{<\infty}\mbox{ with }|A|\le |B|.$$
Note that while the set $A$ is restricted to $\mathcal{S}_\alpha$, the set $B$ is not. 

\begin{thm}\cite[Theorem 1.5]{BC}\label{pmt1}
For every countable ordinal $\alpha$, a basis is $\mathcal{S}_\alpha$-greedy if and only if it is quasi-greedy, $\mathcal{S}_\alpha$-unconditional, and $\mathcal{S}_\alpha$-democratic. 
\end{thm}

Furthermore, \cite[Corollary 1.9 and Theorem 1.10]{BC} state that given countable ordinals $\alpha < \beta$, an $\mathcal{S}_{\beta}$-greedy basis is $\mathcal{S}_{\alpha}$-greedy, while there is an $\mathcal{S}_\alpha$-greedy basis that is not $\mathcal{S}_{\beta}$-greedy. Hence, different countable ordinals give different levels of being quasi-greedy. Due to Theorem \ref{pmt1}, we can be more specific about these levels by asking the following question, which was raised in the last section of \cite{BC}. 

\begin{que}\normalfont\label{mque} Given any pair of countable ordinals $(\alpha,\beta)$, is there a quasi-greedy basis that is 
\begin{itemize}
    \item $\mathcal{S}_\alpha$-unconditional but not $\mathcal{S}_{\alpha+1}$-unconditional, and 
    \item $\mathcal{S}_\beta$-democratic but not $\mathcal{S}_{\beta+1}$-democratic?
\end{itemize}
We call such a basis \textit{$(\alpha, \beta)$-quasi-greedy}. 
\end{que}

In \cite{BC}, the authors constructed 
\begin{itemize}
    \item a $(0,0)$-quasi-greedy basis,
    \item and for each $\alpha < \omega_1$, an $(\infty, \alpha)$-quasi-greedy basis, meaning an unconditional basis that is $\mathcal{S}_\alpha$-democratic but not $\mathcal{S}_{\alpha+1}$-democratic, and
    \item an $(\alpha, \infty)$-quasi-greedy basis, meaning a democratic and quasi-greedy basis that is $\mathcal{S}_\alpha$-unconditional but not $\mathcal{S}_{\alpha+1}$-unconditional.
\end{itemize}
These bases correspond to the filled-in circles in Figure \ref{qgtable}.

\begin{figure}[h]\label{qgtable}

\begin{center}
    \begin{tikzpicture}[scale=0.8]
       
       \draw[->] (0,-.5) -- (10
       ,-.5) node[right] {unconditional};
       \draw[->] (-.5,0) -- (-.5,10) node[right,  rotate=90] {democratic};

        \foreach \x in {0,1,2} 
            \node at (\x,-1)  {\x};
    
        \node at (4,-1) {$\omega$};
        \node at (5,-1) {$\omega$+1};
        \node at (7,-1) {$\omega^2$};
        \node at (10,-1) {$\infty$};
          \node at (3,-1) {$\cdots$};
        \node at (6,-1) {$\cdots$};
        \node at (8,-1) {$\cdots$};
        
        \foreach \x in {0,1,2} 
            \node at (-1,\x)  {\x};
            
        \node at (-1,4) {$\omega$};
        \node at (-1,5) {$\omega$+1};
        \node at (-1,7) {$\omega^2$};
         \node at (-1,8) {$\omega^2$+1};
        \node at (-1,10) {$\infty$};

                \node at (-1, 3) {$\vdots$};
        \node at (-1, 6) {$\vdots$};
    \node at (-1, 9) {$\vdots$};

        \foreach \x in {1,2} 
            \node at  (\x,0)  {$\diamond$} ;
        \fill[fill] (0,0) circle (2pt) ;    

        \node at  (5,0)  {$\diamond$} ;
        \node at  (7,0)  {$\diamond$} ;
        \node at (6, 0) {$\cdots$};
         \node at (8, 0) {$\cdots$};
        \fill[fill] (10,0) circle (2pt) ;

        \foreach \x in {0,1,2} 
            \fill[fill] (\x,10) circle (2pt) ;

        \node at (3, 10) {$\cdots$};
        \fill[fill] (4,10) circle (2pt) ;

        \fill[fill] (5,10) circle (2pt) ;
        \node at (6, 10) {$\cdots$};
        \fill[fill] (7,10) circle (2pt) ;
        \node at (8, 10) {$\cdots$};
        \fill[fill] (10,10) circle (2pt) ;

        \node at (0,1) {$\diamond$} ;
        \node at (1,1) {$\diamond$} ;
        \node at (2,1) {$\diamond$} ;

        \node at (3, 1) {$\cdots$};

        \node at (4,1) {$\diamond$} ;

        \node at (5,1) {$\diamond$} ;
        \node at (6, 1) {$\cdots$};

        \node at (7,1) {$\diamond$} ;
        \fill[fill] (10,1) circle (2pt) ;

        \draw (0,2) circle (2pt) ;
        \node at (1,2) {$\diamond$} ;
        \node at (2,2)  {$\diamond$} ;

        \node at (3, 2) {$\cdots$};
   
        \node at (4,2) {$\diamond$} ;
    
        \node at (5,2) {$\diamond$} ;
        \node at (6, 2) {$\cdots$};
      
        \node at (7,2) {$\diamond$} ;
        \fill[fill] (10,1) circle (2pt) ;
        \draw (0,4) circle (2pt) ;
        
        \draw (1,4) circle (2pt) ;
        \draw (2,4) circle (2pt) ;
      
        \node at (3,4) {$\cdots$};
        \node at (4,4)  {$\diamond$} ;

        \node at (5,4) {$\diamond$} ;
        \node at (6, 4) {$\cdots$};
  
        \node at (7,4) {$\diamond$} ;
        \fill[fill] (10,4) circle (2pt) ;

       \draw (0,5) circle (2pt) ;
        \draw (1,5) circle (2pt) ;
        \draw (2,5) circle (2pt) ;
    
        \node at (3, 5) {$\cdots$};
          \node at (4,5)  {$\diamond$} ;
   
         \node at (5,5)  {$\diamond$} ;
        \node at (6, 5) {$\cdots$};
   
        \node at (7,5) {$\diamond$} ;
        \fill[fill] (10,5) circle (2pt) ;

        \draw (0,7) circle (2pt) ;
        \draw (1,7) circle (2pt) ;
        \draw (2,7) circle (2pt) ;

        \node at (3, 7) {$\cdots$};
        \draw (4,7) circle (2pt) ;

        \draw (5,7) circle (2pt) ;
        \node at (6, 7) {$\cdots$};
        \node at  (7,7)  {$\diamond$} ;
        \fill[fill] (10,7) circle (2pt) ;

                       \draw (0,8) circle (2pt) ;
        \draw (1,8) circle (2pt) ;
        \draw (2,8) circle (2pt) ;
   
        \node at (3, 8) {$\cdots$};
          \draw(4,8) circle (2pt) ;
    
         \draw(5,8) circle (2pt) ;
        \node at (6, 8) {$\cdots$};
        \node at  (7,8)  {$\diamond$} ;
        \fill[fill] (10,8) circle (2pt) ;

        \node at (3, 7) {$\cdots$};

        \node at (0, 3) {$\vdots$};

        \node at (0,6) {$\vdots$};

        \node at (3, 0) {$\cdots$};
        \node at (3, 3) {$\Ddots$};

           \node at (6, 6) {$\Ddots$};

          \node at (9, 9) {$\Ddots$};

\end{tikzpicture}
\end{center}
\caption{Higher-order quasi-greedy bases. The horizontal axis indicates the unconditionality level, while the vertical axis indicates the democracy level. The filled-in circles (\scalebox{0.7}{$\bullet$}) correspond to bases that was already constructed in previous work; the empty circles (\scalebox{0.7}{$\circ$}) correspond to bases that are unknown; the diamonds ($\diamond$) are new bases constructed in this present paper.}
\end{figure}
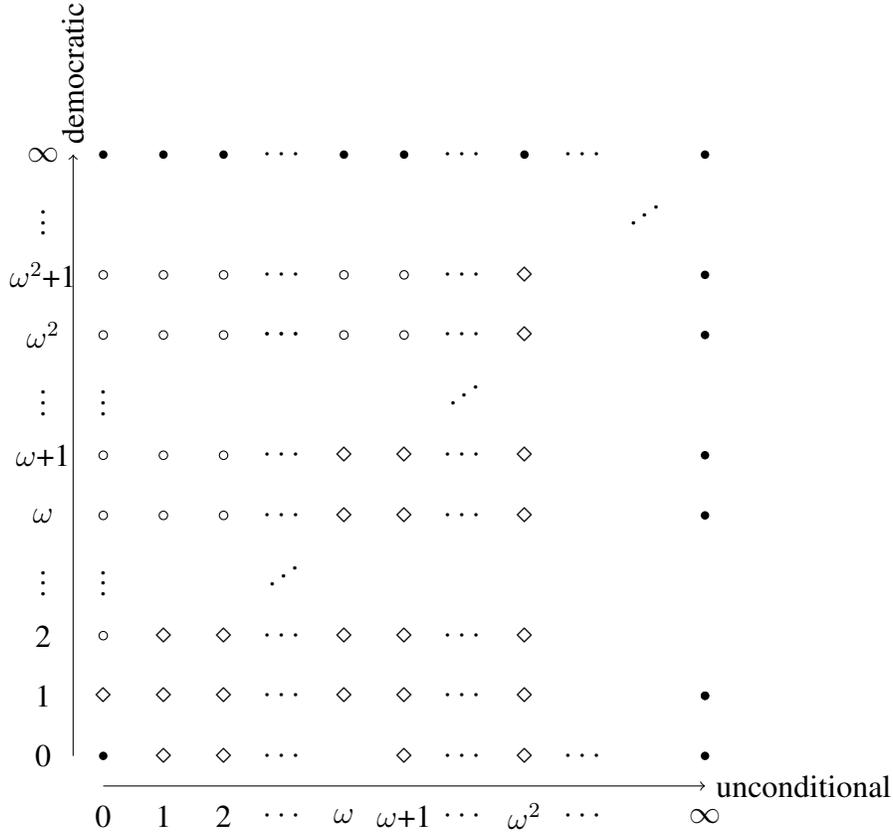

The present paper reports our progress on Question \ref{mque}. For $\beta\le \alpha+1$ and $(\alpha, \beta)\neq (0,0)$, we construct an $(\alpha, \beta)$-quasi-greedy basis. These corresponds to the diamonds in Figure \ref{qgtable}.

All of the Banach spaces we construct are the completion (under a certain norm) of
$c_{00}$, the vector space of finitely supported scalar sequences. The canonical unit vector basis of $c_{00}$, which we denote by $(e_i)_i$, will then always be a normalized Schauder basis of the completion.

\section{The Schreier families and repeated averages}\label{Schreierfam}
Given two sets $A, B\subset \mathbb{N}$ and $m\in \mathbb{N}$, we write $A < B$ to mean $\max A < \min B$ and write $m < A$ or $m\le A$ to mean $m < a$ or $m \le a$, respectively, for all $a\in A$. We also use the convention that $\emptyset < A$ and $A < \emptyset$ for all $A\subset \mathbb{N}$. 

For a countable ordinal $\alpha$, the Schreier family $\mathcal{S}_{\alpha}\subset [\mathbb{N}]^{<\infty}$ is defined recursively as follows \cite{AA}: $$\mathcal{S}_0 \ =\  \{\emptyset\}\cup \bigcup_{n\in \mathbb{N}}\{\{n\}\}.$$
Suppose that $\mathcal{S}_\beta$ has been defined for all $\beta < \alpha$. 

If $\alpha$ is a successor ordinal, i.e., $\alpha = \beta + 1$, then
\begin{equation}\label{er11}\mathcal{S}_\alpha\ =\ \left\{\cup_{i=1}^m E_i\,:\, m \le E_1 < E_2 < \cdots < E_m\mbox{ and }E_i\in \mathcal{S}_\beta, \forall 1\le i\le m\right\}.\end{equation}

If $\alpha$ is a limit ordinal, we choose a sequence of successor ordinals $(\lambda(\alpha, i))_{i=1}^\infty$, which increases to $\alpha$, called an $\alpha$-\textit{approximating sequence}, and put
\begin{equation}\label{er10}
\mathcal{S}_\alpha\ =\ \{E\subset\mathbb{N}\,:\, \exists m\le E, E\in \mathcal{S}_{\lambda(\alpha,m)+1}\}.
\end{equation}
It follows easily and is well known that the families $\mathcal{S}_\alpha$ are \textit{almost increasing} with respect to $\alpha$, meaning that for $0\le \alpha < \beta$, there exists an $N\in \mathbb{N}$ so that
\begin{equation}\label{er12}\{E\in \mathcal{S}_\alpha\,:\, N < E\}\ \subset\ \mathcal{S}_\beta.\end{equation}
It was observed in \cite{TS} that in the recursive definition of $\mathcal{S}_\alpha$, one can choose for a limit ordinal $\alpha$ the $\alpha$-approximating sequence $(\lambda(\alpha, i))$ so that
\begin{equation}\label{er13}\mathcal{S}_{\lambda(\alpha, i)}\ \subset\ \mathcal{S}_{\lambda(\alpha, i+1)}, \mbox{ for }i\in \mathbb{N}.\end{equation}
This choice allows us to rewrite \eqref{er10} as: for each limit ordinal $\alpha$, 
$$\mathcal{S}_\alpha \ =\ \left\{\bigcup_{i=1}^m E_i\,:\, m\le E_1 < E_2 < \cdots < E_m\mbox{ and }E_i\in \mathcal{S}_{\lambda(\alpha, m)}, \forall 1\le i\le m\right\}.$$

From now on, we assume that $\mathcal{S}_\alpha\subset [\mathbb{N}]^{<\infty}$, $\alpha < \omega_1$, is chosen satisfying 
\eqref{er11}, \eqref{er10}, and \eqref{er12}, and that for limit ordinals $\alpha < \omega_1$, the $\alpha$-approximating sequence $(\lambda(\alpha, i))_{i=1}^\infty$ satisfies \eqref{er13}. 

It can be shown by transfinite induction that each Schreier family $\mathcal{S}_\alpha$ is \textit{hereditary} ($F\in \mathcal{S}_\alpha$ and $G\subset F$ imply $G\in \mathcal{S}_\alpha$), \textit{spreading} ($\{m_1, \ldots, m_n\}\in \mathcal{S}_{\alpha}$ and $k_i\ge m_i$, for $i = 1, 2, \ldots, n$, imply $\{k_1, \ldots, k_n\}\in \mathcal{S}_{\alpha}$), and \textit{compact} as a subset of $\{0,1\}^{\mathbb{N}}$ with respect to the product of the discrete topology on $\{0, 1\}$.

Since $\mathcal{S}_\alpha$ is compact, every set in $\mathcal{S}_\alpha$ is contained in some maximal set in $\mathcal{S}_\alpha$. Let $\Max(\mathcal{S}_\alpha)$ be the collection of maximal sets in $\mathcal{S}_\alpha$. In particular, $\Max(\mathcal{S}_\alpha)$ can be described recursively as follows (see \cite[Propositions 2.1 and 2.2]{TS}):

If $\alpha = \beta + 1$, then $A\in \Max(\mathcal{S}_\alpha)$ if and only if there exist $B_1 < B_2 < \cdots < B_{\min A}\in \Max(\mathcal{S}_\beta)$ so that $A = \cup_{i=1}^{\min A} B_i$. Moreover, the sets $B_i\in \Max(\mathcal{S}_\beta)$ are unique. 

If $\alpha$ is a limit ordinal, then $A\in \Max(\mathcal{S}_\alpha)$ if and only if $A\in \Max(\mathcal{S}_{\lambda(\alpha, \min A)+1})$.

\begin{rek}\normalfont
Let us put for a successor ordinal $\alpha = \beta + 1$ and $n\in \mathbb{N}$, $\lambda(\alpha, n) = \beta$. 
Then it follows for any $\alpha < \omega_1$, whether $\alpha$ is a limit or a sucessor ordinal, that for any $A\in \mathcal{S}_\alpha$, there are $A_1 < A_2 < \cdots < A_{\min A}$ in $\mathcal{S}_{\lambda(\alpha, \min A)}$ (possibly some of the $A_i$ could be empty) so that
$$A\ =\ \bigcup_{i=1}^{\min A}A_i.$$
Furthermore, for any $A\in \Max(\mathcal{S}_\alpha)$, there are unique $A_1 < A_2 < \cdots < A_{\min A}$ in $\Max(\mathcal{S}_{\lambda(\alpha, \min A)})$ so that
\begin{equation}\label{er14}A\ =\ \bigcup_{i=1}^{\min A} A_i.\end{equation}
We call \eqref{er14} the recursive representation of $A\in \Max(\mathcal{S}_\alpha)$. 
\end{rek}

We now define the hierarchy of repeated averages which were introduced in \cite{AA}, and record their properties.

For every $\alpha < \omega_1$ and any $A\in \Max(\mathcal{S}_\alpha)$, we will define a vector $$x_{(\alpha, A)}\ =\ \sum_{i=1}^\infty x_{(\alpha, A)}(i)e_i\in c_{00}$$ having nonnegative coefficients. 

If $\alpha = 0$ and $i\in \mathbb{N}$, then 
$$x_{(0, \{i\})}\ =\ e_i.$$
Assume that $x_{(\beta, B)}$ has been defined for all $\beta < \alpha$ and $B\in \Max(\mathcal{S}_\beta)$. Let $A = \bigcup_{i=1}^{\min A} A_i$, with $A_1 < A_2 < \cdots < A_{\min A}$ in $\Max(\mathcal{S}_{\lambda(\alpha, \min A)})$ the (unique) recursive representation of $A\in \Max(\mathcal{S}_\alpha)$. Then we put
$$x_{(\alpha, A)}\ =\ \frac{1}{\min A}\sum_{i=1}^{\min A}x_{(\lambda(\alpha, \min A), A_i)}.$$

Let $M\subset \mathbb{N}$ be infinite and $\alpha  <\omega_1$. Then define the sets $A(\alpha, M, 1) < A(\alpha, M, 2) < A(\alpha, M, 3) < \cdots$ in $\Max(\mathcal{S}_\alpha)$ so that
$$M\ =\ \bigcup_{i=1}^\infty A(\alpha, M, i).$$
For $i\in \mathbb{N}$, we put
$$x_{(\alpha, M, i)}\ =\ x_{(\alpha, A(\alpha, M, i))}.$$

The following properties can be shown by transfinite induction (cf.\ \cite{AMT}) for all $\alpha < \omega_1$:
\begin{itemize}
\item[(P1)] Each $x_{(\alpha, A)}$ is a convex combination of the standard unit vector basis of $c_{00}$, for all $A\in \Max(\mathcal{S}_\alpha)$.
\item[(P2)] The nonzero coefficients of $x_{(\alpha, A)}$ are decreasing.  
\item[(P3)] $\supp (x_{(\alpha, A)}) = A$, for all $A\in \Max(\mathcal{S}_\alpha)$.
\item[(P4)] If $A_1 < A_2 < \cdots$ are in $\Max(\mathcal{S}_\alpha)$, then
$$x_{(\alpha, M, i)}\ =\ x_{(\alpha, A_i)}, \mbox{ for all }i\in \mathbb{N},$$
where $M = \cup_{i=1}^\infty A_i$. 
\end{itemize}

We will later need the following observation.  

\begin{lem}\label{boundby6}
Let $\alpha < \omega_1$ and $N\in \mathbb{N}$. Let $A_1 < A_2 < \cdots < A_N$ be in $\Max(\mathcal{S}_\alpha)$ and $F\in \mathcal{S}_\alpha$. Then 
$$\sum_{j\in F}\sum_{i=1}^N x_{(\alpha, A_i)}(j)\ \le\ 6.$$
\end{lem}

\begin{proof}
For $\alpha = 0$, our claim is trivially true. Assume that for all $\gamma < \alpha$, our claim is correct. Let $A_1 < A_2 < \cdots < A_N$ be in $\Max(\mathcal{S}_\alpha)$. Thus, for $i = 1, 2, \ldots, N$, we write
$$x_{(\alpha, A_i)}\ =\ \frac{1}{\min A_i}\sum_{s = 1}^{\min A_i}x_{(\lambda(\alpha, \min A_i), A_{(i,s)})},$$
where $A_{(i,1)} < A_{(i, 2)} < \cdots < A_{(i, \min A_i)}$ are in $\Max(\mathcal{S}_{\lambda(\alpha, \min A_i)})$ and $A_i = \cup_{s=1}^{\min A_i} A_{(i, s)}$.

Let $F\in \mathcal{S}_\alpha$, which we can assume to be in $\Max(\mathcal{S}_\alpha)$ and write $F$ as
$$F\ =\ \bigcup_{i=1}^{\min F} F_i,\mbox{ where }F_1 < F_2 < \cdots < F_{\min F}\mbox{ are in }\Max(\mathcal{S}_{\lambda(\alpha, \min F)}).$$

Without loss of generality, assume that $\min F\le \max A_1$. Note that for $i = 1, 2, \ldots, N$, we have
$$\min A_{i+1} \ \ge\ 1 + \max A_i \ \ge\ 1 + \min A_i + |A_i| - 1\ \ge\ 2\min A_i.$$
It follows that for all $i\ge 4$, 
$$\min A_i\ \ge\ 2^{i-2}\min A_2 \ >\ 2^{i-2}\min F.$$

We deduce that
\begin{align*}
    \sum_{t\in F}\sum_{i=1}^N x_{(\alpha, A_i)}(t)&\ \le\ 3 + \sum_{j = 1}^{\min F}\sum_{i=4}^N\frac{1}{\min A_i}\sum_{t\in F_j}\sum_{s=1}^{\min A_i}x_{(\lambda(\alpha,\min A_i), A_{(i,s)})}(t)\mbox{ by (P1)}\\
    &\ \le\ 3 + \frac{1}{\min F}\sum_{j=1}^{\min F}\sum_{i=4}^{N}2^{2-i}\sum_{t\in F_j}\sum_{s=1}^{\min A_i}x_{(\lambda(\alpha,\min A_i), A_{(i,s)})}(t).
\end{align*}
For $j = 1, 2, \ldots, \min F$ and $i = 4, 5, \ldots, N$, we have 
$$F_j\ \in\ \mathcal{S}_{\lambda(\alpha, \min F)}\ \subset\ \mathcal{S}_{\lambda(\alpha, \min A_i)}.$$
The inductive hypothesis gives 
$$\sum_{t\in F_j}\sum_{s=1}^{\min A_i}x_{(\lambda(\alpha,\min A_i), A_{(i,s)})}(t)\ \le\ 6.$$
Hence, 
$$\sum_{t\in F}\sum_{i=1}^N x_{(\alpha, A_i)}(t)\ \le\ 3 + \frac{1}{\min F}\sum_{j=1}^{\min F}\sum_{i=2}^{\infty}2^{-i}6\ =\ 6,$$
which finishes the proof. 
\end{proof}

\section{Construction of an $(\alpha, \alpha+1)$-quasi-greedy basis}\label{(alpha, alpha+1)}
In this section, we construct an $(\alpha, \alpha+1)$-quasi-greedy basis.

\subsection{The gauge functions $\boldsymbol{\psi}$ and $\boldsymbol{\phi}$, for general $\boldsymbol{\alpha < \omega_1}$}
Let $\alpha\in [1, \omega_1)$ and $m\in \mathbb{N}$, we define the strictly increasing sequence $(s_{(\alpha, m)}(i))_{i=0}^\infty\subset \mathbb{N}$ by 
$$s_{(\alpha, m)}(0) \ =\ m, \quad A(\alpha, m, i)\ :=\ [s_{(\alpha, m)}(i-1), s_{(\alpha, m)}(i)-1]\in \Max(\mathcal{S}_\alpha), \mbox{ for }i\in \mathbb{N},$$
and  thus $(A(\alpha, m, i))_{i\in \mathbb{N}} = ([s_{(\alpha, m)}(i-1), s_{(\alpha, m)}(i)-1])_{i\in \mathbb{N}}$ is a partition of the set $\{m, m+1, m+2, \ldots\}$. From the construction of $\mathcal{S}_\alpha$, it follows that
\begin{equation}\label{re50}s_{(\alpha+1, m)}(1)\ =\ s_{(\alpha, m)}(m).\end{equation}
Then we define $\theta_{(\alpha, m)}: [m, \infty)\rightarrow \mathbb{R}$, by letting $\theta_{(\alpha, m)}(s_{(\alpha, m)}(i)) = \log m + i$, for $i = 0, 1, 2, \ldots$, and defining $\theta_{(\alpha, m)}(x)$, for other values $x$, by linear interpolation.  

\begin{prop}\label{ppp0}
For $1\le \alpha < \omega_1$ and $m\ge 10^5$,
\begin{equation}\label{e13}\theta_{(\alpha, m)}(x)\ \leq\ \sqrt[4]{x}, \mbox{ for all }x\in [m, \infty).\end{equation}
\end{prop}

\begin{proof}
For each $i\in \mathbb{N}$, $s_{(\alpha, m)}(i) \geq 2^{i}m$, because $\mathcal{S}_1\subset \mathcal{S}_\alpha$, and thus,
$$\frac{\theta_{(\alpha, m)}(s_{(\alpha, m)}(i))}{\sqrt[4]{s_{(\alpha, m)}(i)}}\ \leq\ \frac{\log m + i}{\sqrt[4]{2^{i}m}}\ =\ \frac{\log m}{\sqrt[4]{2^{i}m}} + \frac{i}{\sqrt[4]{2^{i}m}}\ <\ \frac{\log m}{\sqrt[4]{m}} + \frac{3}{\sqrt[4]{m}}.$$
Therefore, for $m\ge 10^5$, $\theta_{(\alpha, m)}(s_{(\alpha, m)}(i)) \leq \sqrt[4]{s_{(\alpha, m)}(i)}$. 
Then linear interpolation and the concavity of $\sqrt[4]{x}$ guarantee \eqref{e13}. 
\end{proof}

\begin{prop}\label{ppp1}
For $1\le \alpha < \omega_1$ and $m\ge 10^5$,
the function $\theta_{(\alpha, m)}^2(x)/x$ is strictly decreasing on $[m, \infty)$.
\end{prop}

\begin{proof}
Let $f(x):= \theta_{(\alpha, m)}^2(x)/x$, for $x\ge m$. Since $f(x)$ is continuous, it suffices to show that for $m\ge 10^5$ and $i\in \mathbb{N}$, $f(x)$ is decreasing on $(s_{(\alpha, m)}(i-1), s_{(\alpha, m)}(i))$.

We have $$f'(x) \ =\ \frac{2\theta_{(\alpha, m)}(x)\theta'_{(\alpha, m)}(x) x-\theta^2_{(\alpha, m)}(x)}{x^2}\ =\ \frac{\theta_{(\alpha, m)}(x)(2x\theta'_{(\alpha, m)}(x)-\theta_{(\alpha, m)}(x))}{x^2}.$$
We need to verify that 
\begin{equation}\label{e14}2x\theta'_{(\alpha, m)}(x)\ <\ \theta_{(\alpha, m)}(x), \mbox{ for all }x\in (s_{(\alpha, m)}(i-1), s_{(\alpha, m)}(i)).\end{equation}
Write $x = (1-t)s_{(\alpha, m)}(i-1) + ts_{(\alpha, m)}(i)$ for some $t\in (0,1)$. Then \eqref{e14} is equivalent to
$$\frac{2((1-t)s_{(\alpha, m)}(i-1) + ts_{(\alpha, m)}(i))}{s_{(\alpha, m)}(i)-s_{(\alpha, m)}(i-1)}\ <\ (1-t)(\log m + i-1) + t(\log m + i).$$
Equivalently, 
$$2s_{(\alpha, m)}(i-1) \ <\ (\log m+i-1-t)(s_{(\alpha, m)}(i)-s_{(\alpha, m)}(i-1)),$$
which is clearly true for $m\geq 10^5$ because $s_{(\alpha, m)}(i)\geq 2s_{(\alpha, m)}(i-1)$.
\end{proof}

Our goal is to define a map $\psi: [0,\infty)\rightarrow \mathbb{R}$  and two strictly increasing subsequences $M_1$ and $M_2$ of $\mathbb{N}$ satisfying the following properties
\begin{enumerate}
\item[a)]\label{propa} $\psi(0) = 0$, $\psi(1) = 1$, $\psi(x)\nearrow \infty$, $\psi(x)/x\searrow 0$ as $x\rightarrow\infty$;
\item[b)] $\psi$ is concave on $[1,\infty)$;
\item[c)] $\psi(x)\leq \sqrt{x}$ for all $x\geq 1$, and for each $m\in M_1$, we have
$$\psi(x) \ =\ \sqrt{x}, \mbox{ for all }x\in [\log m, m];$$
\item[d)] for each $n\in M_2$, we have
$$\theta^2_{(\alpha+1, n)}(x)\ \leq\  \psi(x)  \ \leq\ 2\theta^2_{(\alpha+1, n)}(x), \mbox{ for all }x\in [n, s_{(\alpha+1, n)}(n)].$$
\end{enumerate}

To obtain such a function $\psi$ and sets $M_1, M_2$, we choose integers $m_0 < m_1 < n_1 < m_2 < n_2 < m_3 < n_3 < \cdots$ such that 
\begin{itemize}
\item $m_0 = 1$ and $m_1\ge 10^5$;
\item for any $i\in \mathbb{N}$, 
\begin{equation}\label{e17}s_{(\alpha, m_i)}(m_i) \ =\ s_{(\alpha+1, m_i)}(1)\ <\ \log n_i\ <\ s_{(\alpha+2, n_i)}(1)\ <\ \sqrt{\log m_{i+1}}.\end{equation}
\end{itemize}
Put 
$$\widetilde{\psi}(x)  \ =\ \begin{cases} 1, &\mbox{ if }x = 1,\\  \sqrt{x}, &\mbox{ if }x\in [\log m_i, m_i],\mbox{ for }i = 1, 2, 3, \ldots,\\ \theta^2_{(\alpha+1, n_i)}(x), &\mbox{ if }x\in [n_i, s_{(\alpha+2, n_i)}(1)], \mbox{ for }i = 1, 2, 3, \ldots,\\
\mbox{ by linear interpolation}, &\mbox{ otherwise}.\end{cases}$$

Let $M_1 = \{m_j: j\in \mathbb{N}\}$ and $M_2 = \{n_j:j\in\mathbb{N}\}$. Thanks to Proposition \ref{ppp0}, $\widetilde{\psi}(x)$ satisfies c). By construction, $\widetilde{\psi}(x)$ satisfies d). Furthermore, \eqref{e17} gives 
$$\widetilde{\psi}(m_i) \ =\ \sqrt{m_i}\ <\ \log^2 n_i\ =\ \widetilde{\psi}(n_i)$$
and
\begin{align*}
    \widetilde{\psi}(s_{(\alpha+2, n_i)}(1))\ =\ \theta^2_{(\alpha+1, n_i)}(s_{(\alpha+2, n_i)}(1))&\ \le\ \sqrt{s_{(\alpha+2, n_i)}(1)}\\
    &\ <\ \sqrt{\log m_{i+1}}\ =\  \widetilde{\psi}(\log m_{i+1});
\end{align*}
hence, $\widetilde{\psi}(x)\nearrow \infty$. Finally, we verify that 
$\widetilde{\psi}(x)/x\searrow 0$. By Propositions \ref{ppp0} and \ref{ppp1}, $\lim_{x\rightarrow\infty}\widetilde{\psi}(x)/x = 0$ with $\widetilde{\psi}(x)/x$ decreasing on $[\log m_i, m_i]$ and $[n_i, s_{(\alpha+2, n_i)}(1)]$, and
\begin{align*}
    &\frac{\widetilde{\psi}(n_i)}{n_i}\ =\ \frac{\log^2 n_i}{n_i}\ <\ \frac{1}{\log n_i} \ < \ \frac{1}{\sqrt{m_i}}\ = \ \frac{\widetilde{\psi}(m_i)}{m_i},\\
    &\frac{\widetilde{\psi}(\log m_{i+1})}{\log m_{i+1}}\ =\ \frac{1}{\sqrt{\log m_{i+1}}}\ <\ \frac{1}{s_{(\alpha+2, n_i)}(1)}\ <\ \frac{\widetilde{\psi}(s_{(\alpha+2, n_i)}(1))}{s_{(\alpha+2, n_i)}(1)}.
\end{align*}

Recall from \cite[pg.\ 46]{DKOSZ} that a function $g(x): [1,\infty)\rightarrow \mathbb{R}^+$ is called \textit{fundamental} if 
it is increasing and $x\mapsto g(x)/x$ is decreasing. Our function $\widetilde{\psi}(x)$ is fundamental. 
By \cite[Lemma 7]{DKOSZ}, there exists the smallest concave fundamental function $\psi: [1,\infty)\rightarrow \mathbb{R}^+$ that dominates $\widetilde{\psi}$, i.e.,  
\begin{equation*}\widetilde{\psi}(x)\ \leq\ \psi(x)\ \leq\ 2\widetilde{\psi}(x).\end{equation*} 
Since $\sqrt{x}$ is a concave function which dominates $\widetilde{\psi}(x)$, it follows that $\psi(x)\le \sqrt{x}$, for $x\ge 1$. On $[0,1]$, we set $\psi(x) = x$. Therefore, $\psi$ satisfies all of a), b), c), and d).

Now define $\phi(x) = \sqrt{\psi(x)}$. It follows that $\phi$ satisfies a) and b). Moreover, we deduce that
\begin{enumerate}
\item[e)] for all $x\in [1,\infty)$, 
\begin{equation}\label{re31}\phi(x)\leq \sqrt[4]{x},\end{equation}
and for each $m\in M_1$, 
\begin{equation}\label{re32}\phi(x) \ =\ \sqrt[4]{x}, \mbox{ for all }x\in [\log m, m];\end{equation}
\item[f)] for each $n\in M_2$, we have
$$\theta_{(\alpha+1, n)}(x)\ \leq\ \phi(x)\ \leq\ \sqrt{2}\theta_{(\alpha+1, n)}(x), \mbox{ for all }x\in [n, s_{(\alpha+2, n)}(1)].$$
\end{enumerate}

\subsection{An $\boldsymbol{(\alpha, \alpha+1)}$-quasi-greedy basis for $\boldsymbol{\alpha\ge 0}$}Recall from Section \ref{Schreierfam} that given $A\in \Max(\mathcal{S}_{\alpha})$,  $x_{(\alpha, A)}$ is the repeated average of order $\alpha$ with support $A$. For $x = (x_i)_{i=1}^\infty\in c_{00}$, let 
\begin{align*}
&\|x\|_1\ =\ \sup\left\{\frac{\phi(s)}s \sum_{j=1}^s\sum_{i\in A_j} x_{(\alpha,A_j)}(i) |x_{\pi(i)}| : \begin{matrix} s\le A_1<\ldots< A_s \text{ in }\Max({\mathcal S}_\alpha)\\
          \pi: \bigcup_{j=1}^s A_j\to {\mathbb N}\text{ strictly increasing},\\ \text{with } \pi\big(\bigcup_{j=1}^s A_j\big)\in {\mathcal S}_1
\end{matrix}\right\},\\
&\|x\|_2\ =\ \max_{m\in M_2}\sum_{k=1}^m \phi(s_{(\alpha+1, m)}(k-1))\sum_{i\in A(\alpha+1, m, k)}x_{(\alpha+1, A(\alpha+1, m, k))}(i)|x_i|,\\
&\|x\|_3 \ =\ \max_{m\in M_1}\left(\sum_{k=1}^{m}(\psi(k)-\psi(k-1))\sum_{i\in A(\alpha, m, k)} x_{(\alpha, A(\alpha, m, k))}(i)x_i^2\right)^{1/2}, \mbox{ and}\\
&\|x\|_4 \ =\ \max_{m\in M_1}\max_{i_0\in \mathbb{N}}\left|\sum_{k=1}^m(\phi(k)-\phi(k-1))\sum_{i\in A(\alpha, m, k), i\le i_0}x_{(\alpha, A(\alpha, m, k))}(i)x_i\right|.\\
\end{align*}

Let $X$ be the completion of $c_{00}$ with respect to the norm $\|\cdot\| := \max_{1\leq i\leq 4}\|\cdot\|_i$. Then $(e_i)_i$ is normalized. Specifically, the norm of $e_i$ is realized by setting $s =1, A_1 = A(\alpha, 1, 1) = \{1\}$, and $\pi(1) = i$ in the definition of $\|\cdot\|_1$. (Note that $\{1\}\in \Max(\mathcal{S}_\alpha)$ for all $\alpha$.) This also shows that $\|(x_i)_i\|_1\ge \max_{i\ge 1}|x_i|$.

\begin{rek}\normalfont
When $\alpha = 0$, we can use a slightly simpler norm $\|\cdot\|_1$ without the map $\pi$. In particular,
\begin{align*}
    \|x\|_1&\ =\ \max_{F\in \mathcal{S}_1, F\neq \emptyset} \frac{\phi(|F|)}{|F|}\sum_{i\in F}|x_i|,\\
    \|x\|_2&\ =\ \max_{m\in M_2}\sum_{k=1}^m \frac{\phi(|A(1, m, k)|)}{|A(1, m, k)|}\sum_{i\in A(1, m, k)}|x_i|,\\
    \|x\|_3&\ =\ \max_{m\in M_1}\left(\sum_{k=1}^{m}(\psi(k)-\psi(k-1))x_{k+m-1}^2\right)^{1/2}, \mbox{ and }\\
    \|x\|_4&\ =\ \max_{m\in M_1}\max_{1\le j\le m}\left|\sum_{k=1}^{j}(\phi(k)-\phi(k-1))x_{k+m-1}\right|.
\end{align*}
Let us briefly explain why the case $\alpha = 0$ does not require the map $\pi$. For every nonempty $F \in [\mathbb{N}]^{<\infty}$, the set of the largest $\lfloor (|F|+1)/2\rfloor$ integers in $F$ is an $\mathcal{S}_1$-set, which can be decomposed into at least $|F|/2$ maximal $\mathcal{S}_0$-sets (or singletons). However, for $\alpha\ge 1$, there may not exist an $\mathcal{S}_{\alpha+1}$-subset of $F$ that can be decomposed into $|F|/2$ maximal $\mathcal{S}_\alpha$-sets. For example, if $\alpha = 1$, the set $F = \{10, 11, 12, \ldots, 18\}$ has no subset in $\Max(\mathcal{S}_1)$. This distinction between the cases $\alpha = 0$ and $\alpha\ge 1$ necessitates the introduction of the map $\pi$ when $\alpha \ge 1$.
\end{rek}

\subsection{$\boldsymbol{\mathcal{S}_{\alpha+1}}$-democratic but not $\boldsymbol{\mathcal{S}_{\alpha+2}}$-democratic}
For $\alpha < \omega_1$ and a set $E\in [\mathbb{N}]^{<\infty}$, let $t_\alpha(E)$ be the largest nonnegative integer such that there are sets $A_1 < A_2 < \cdots < A_{t_\alpha(E)}$ in $\Max(\mathcal{S}_\alpha)$ with $\bigcup_{i=1}^{t_{\alpha}(E)}A_i\subset E$.

\begin{lem}\label{l10}
For $\alpha < \omega_1$ and $E\in \mathcal{S}_{\alpha+1}$, we have $t_\alpha(E)\le m$, 
where $m$ is the smallest positive integer such that $[m, m+|E|-1]\in \mathcal{S}_{\alpha+1}$. 
\end{lem}

\begin{proof}
Let $m\in \mathbb{N}$ be the smallest positive integer such that $[m, m+|E|-1]\in \mathcal{S}_{\alpha+1}$. 
Suppose, for a contradiction, that $m < t_\alpha(E)$. Since $E\in \mathcal{S}_{\alpha+1}$, 
we have $t_\alpha(E) \le \min E$, and thus, $m < \min E$. It follows from $m < \min E$ and the spreading property that
\begin{equation}\label{re51}t_\alpha([m, m+|E|-1]) \geq t_\alpha(E).\end{equation}
Since $[m, m+|E|-1]\in \mathcal{S}_{\alpha+1}$, 
\begin{equation}\label{re52}m \ \ge\ t_{\alpha}([m, m+|E|-1]).\end{equation}
From \eqref{re51} and \eqref{re52}, we obtain $m \geq t_{\alpha}(E)$. This contradicts our supposition. 
\end{proof}

Our next several results establish bounds for $\|1_E\|_i$, with or without the condition $E\in \mathcal{S}_{\alpha+1}$, such that the bounds depend only on $|E|$. 

\begin{lem}\label{l1}
For a nonempty set $E\in [\mathbb{N}]^{<\infty}$, it holds that 
$$\|1_E\|_1\ \leq\ 6\phi(m+1),$$
where $m$ is the least positive integer such that $[m, m+|E|-1]\in \mathcal{S}_{\alpha+1}$.
\end{lem}

\begin{proof}
Let $s\le A_1 < A_2 < \cdots < A_s$ be in $\Max(\mathcal{S}_\alpha)$ and let $\pi: \bigcup_{i=1}^s A_i\rightarrow \mathbb{N}$ be a strictly increasing map with $\pi(\bigcup_{i=1}^s A_i)\in \mathcal{S}_1$. We have
\begin{align*}&\frac{\phi(s)}{s}\sum_{j=1}^s\sum_{i\in A_j} x_{(\alpha, A_j)}(i)|{(1_E)}_{\pi(i)}|\ =\ \frac{\phi(s)}{s}\sum_{j=1}^s\sum_{i\in A_j\cap \pi^{-1}(E)} x_{(\alpha, A_j)}(i).
\end{align*}

If $s < m+1$, then 
$$\frac{\phi(s)}{s}\sum_{j=1}^s\sum_{i\in A_j\cap \pi^{-1}(E)} x_{(\alpha, A_j)}(i)\ \le\ \phi(s)\ \le\ \phi(m+1).$$

Suppose that $s\ge m+1$. Let $m'$ be the smallest positive integer such that
$$[m', m'+|\pi^{-1}(E)|-1]\ \in\ \mathcal{S}_{\alpha+1}.$$
Then $m'\le m$ because $|\pi^{-1}(E)|\le |E|$. 
Write $\pi^{-1}(E) = \bigcup_{i=1}^{t_\alpha(\pi^{-1}(E))+1}B_{i}$ for $B_i\in \mathcal{S}_{\alpha}$. We have
\begin{align*}
&\frac{\phi(s)}{s}\sum_{j=1}^s\sum_{i\in A_j\cap \pi^{-1}(E)} x_{(\alpha, A_j)}(i)\\
&\ =\ \frac{\phi(s)}{s}\sum_{\ell=1}^{t_\alpha(\pi^{-1}(E))+1}\sum_{j=1}^s\sum_{i\in A_j\cap B_\ell} x_{(\alpha, A_j)}(i)\\
&\ \leq\ \frac{\phi(s)}{s}6(t_\alpha(\pi^{-1}(E))+1) \mbox{\quad (by Lemma \ref{boundby6})}\\
&\ \leq\ 6\frac{\phi(s)}{s}(m'+1) \mbox{\quad (by Lemma \ref{l10} applied to $\pi^{-1}(E)\subset \bigcup_{i=1}^sA_i\in \mathcal{S}_{\alpha+1}$)}\\
&\ \leq\ 6\frac{\phi(s)}{s}(m+1)\ \leq\ 6\phi(m+1) \mbox{\quad (by Property a) of $\phi$)}.
\end{align*}
\end{proof}
We need the following lemma to prove a lower bound for $\|1_E\|_1$. 

\begin{lem}\label{al1}
Let $1\le \alpha < \omega_1$ and $m\in \mathbb{N}$. Define $p_1$ and $p_2$ to be the smallest positive integers such that $[p_1, p_1 + m - 1]\in \mathcal{S}_\alpha$ and $[p_2, p_2 + \lfloor (m+1)/2\rfloor - 1]\in \mathcal{S}_\alpha$, respectively. Then $p_2\ge p_1/2$.    
\end{lem}

\begin{proof}
If $\alpha = 1$, then $p_2  = \lfloor (m+1)/2\rfloor \ge m/2 = p_1/2$.

Assume that $\alpha\ge 2$ and that our claim is true for all $\beta < \alpha$. If $p_1 \le 2$, then $p_2\ge 1\ge p_1/2$. So we can assume that $p_1\ge 3$. It follows from the definition of $p_2$ that there are $A_1 < A_2 < \cdots < A_{p_2}$ in $\mathcal{S}_{\lambda(\alpha, p_2)}$ so that
$$[p_2, p_2 + \lfloor (m+1)/2\rfloor - 1]\ =\ \bigcup_{i=1}^{p_2} A_i.$$
Put $A'_i = A_i + 1 = \{a+1: a\in A_i\}\in \mathcal{S}_{\lambda(\alpha, p_2)}\subset \mathcal{S}_{\lambda(\alpha, p_2+1)}$. Then
\begin{equation}\label{ee10}[p_2+1, p_2 + m]\ =\ \bigcup_{i=1}^{p_2} A_i'\cup \underbrace{[p_2 + \lfloor (m+1)/2\rfloor + 1, p_2 + m]}_{\in \mathcal{S}_1\subset \mathcal{S}_{\lambda(\alpha, p_2 +1)}}\ =\ \bigcup_{i=1}^{p_2+1}A'_i,\end{equation}
with $A'_{p_2+1} = [p_2+\lfloor (m+1)/2\rfloor + 1, p_2 + m]$. It follows from \eqref{ee10} that
$[p_2+1, p_2 + m]\in \mathcal{S}_\alpha$. The minimality of $p_1$ gives that $p_1\le p_2 + 1$. Hence,
$$\frac{p_1}{2}\ \le\ p_1-1\ \le\ p_2.$$
\end{proof}

\begin{prop}\label{rp50}
For a nonempty set $E\in [\mathbb{N}]^{<\infty}$, it holds that 
$$\|1_E\|_1\ \ge\ \frac{1}{6}\phi(m+1),$$
where $m$ is the least positive integer such that $[m, m+|E|-1]\in \mathcal{S}_{\alpha+1}$.    
\end{prop}
\begin{proof}
If $m\le 5$, we trivially have
$$\|1_E\|_1\ \ge\ 1\ \ge\ \frac{\phi(6)}{6}\ \ge\ \frac{1}{6}\phi(m+1).$$
Assume that $m \ge 6$.
Let $E'$ be the set containing the largest $\lfloor (|E|+1)/2\rfloor$ elements of $E$. Then 
$$\min E'\ \ge\ |E|-|E'| + 1\ =\ |E| - \left\lfloor \frac{|E|+1}{2}\right\rfloor + 1\ \ge\ \left\lfloor \frac{|E|+1}{2}\right\rfloor\ =\ |E'|.$$
Therefore, $E'\in \mathcal{S}_1$. Choose $p$ to be the smallest positive integer such that $[p, p+\lfloor (|E|+1)/2\rfloor-1]\in \mathcal{S}_{\alpha+1}$. 

If $p = 1$, then $|E|$ is $1$ or $2$. If the former, $m = 1$; if the latter, $m = 2$. Both cases contradict $m \ge 6$.

If $p\ge 2$, the definition of $p$ implies that $[p-1, p+\lfloor (|E|+1)/2\rfloor-2]\notin \mathcal{S}_{\alpha+1}$. 
Choose $q \ge p-1$ such that $F := [p-1, q]\in \Max(\mathcal{S}_{\alpha+1})$. Since 
$$|F| \ =\ |[p-1, q]|\ <\ \left|\left[p-1, p+\left\lfloor \frac{|E|+1}{2}\right\rfloor-2\right]\right|\ =\ \left\lfloor \frac{|E|+1}{2}\right\rfloor\ =\ |E'|,$$
we can define $\pi: F\rightarrow E'$ to be a strictly increasing map. Write $F = \bigcup_{i=1}^{p-1}A_i$ with $A_i\in \Max(\mathcal{S}_\alpha)$. We have
\begin{align*}
\|1_E\|_1\ \ge\ \frac{\phi(p-1)}{p-1}\sum_{j=1}^{p-1}\sum_{i\in A_j} x_{(\alpha, A_j)}(i)\ =\ \phi(p-1).
\end{align*}
By Lemma \ref{al1}, 
$$\phi(p-1)\ \ge\ \phi\left(\frac{m}{2}-1\right)\ \ge\ \phi\left(\frac{m+1}{6}\right)\ \ge\ \frac{1}{6}\phi(m+1).$$
This completes our proof.
\end{proof}

\begin{lem}\label{l2}
For $E\in \mathcal{S}_{\alpha+1}$, it holds that
$$\|1_E\|_2\ \leq\ 6\phi\left(m+1\right),$$
where $m$ is the least positive integer such that $[m, m+|E|-1]\in \mathcal{S}_{\alpha+1}$.
\end{lem}

\begin{proof}
Pick $E\in \mathcal{S}_{\alpha+1}$ and $n\in M_2$. 
We have
\begin{align*}
&T(\alpha, n, E)\ :=\ \sum_{k=1}^n \phi(s_{(\alpha+1, n)}(k-1))\sum_{i\in A(\alpha+1, n, k)}x_{(\alpha+1, A(\alpha+1, n, k))}(i)|(1_E)_i|\\
&\ =\ \sum_{k=1}^n \phi(s_{(\alpha+1, n)}(k-1))\sum_{i\in A(\alpha+1, n, k)\cap E}x_{(\alpha+1, A(\alpha+1, n, k))}(i).
\end{align*}
By Lemma \ref{boundby6},
$$\sum_{k=1}^n \sum_{i\in A(\alpha+1, n, k)\cap E}x_{(\alpha+1, A(\alpha+1, n, k))}(i)\ \le\ 6,$$
so the concavity of $\phi$ implies that 
$$
T(\alpha, n, E)\ \le\  6\phi\left(\frac{1}{6}\sum_{k=1}^ns_{(\alpha+1, n)}(k-1)\sum_{i\in A(\alpha+1, n, k)\cap E}x_{(\alpha+1, A(\alpha+1, n, k))}(i)\right).
$$
Write
$$A(\alpha+1, n, k) = \bigcup_{j=1}^{s_{(\alpha+1, n)}(k-1)} B_{k,j}, \mbox{ with }B_{k,j}\in \Max(\mathcal{S}_\alpha)$$
and $E = \bigcup_{\ell=1}^{t_\alpha(E)+1}A_\ell$, with $A_\ell\in \mathcal{S}_\alpha$ to have
$${x_{(\alpha+1, A(\alpha+1, n, k))}}\ =\ \frac{1}{s_{(\alpha+1, n)}(k-1)}\sum_{j=1}^{s_{(\alpha+1, n)}(k-1)}x_{(\alpha, B_{k,j})},$$
and thus,
\begin{align*}
T(\alpha, n, E)&\ \le\  6\phi\left(\frac{1}{6}\sum_{\ell=1}^{t_\alpha(E)+1}\sum_{k=1}^n \sum_{j=1}^{s_{(\alpha+1, n)}(k-1)}\sum_{i\in B_{k,j}\cap A_\ell}x_{(\alpha, B_{k,j})}(i)\right)\\
&\ \le\ 6\phi\left(\frac{1}{6}\sum_{\ell=1}^{t_\alpha(E)+1}6\right) \quad \mbox{ (by Lemma \ref{boundby6})}\\
&\ =\ 6\phi(t_\alpha(E)+1)\ \le\ 6\phi(m+1)\quad \mbox{ (by Lemma \ref{l10})},
\end{align*}
as desired.
\end{proof}

\begin{lem}\label{l3}
For $E\in [\mathbb{N}]^{<\infty}$, it holds that
\begin{align}
\|1_E\|_3&\ \leq\ \sqrt{6}\phi(m+1), \mbox{ and}\label{e6}\\
\|1_E\|_4&\ \leq\ 6\phi(m+1)\label{e7},
\end{align}
where $m$ is the least positive integer such that $[m, m+|E|-1]\in \mathcal{S}_{\alpha+1}$.
\end{lem}

\begin{proof}
We shall prove \eqref{e6} only since the same proof applies to \eqref{e7}. 
Let $n\in M_1$ and 
$$T(n, E) \ :=\ \sum_{k=1}^{n}(\psi(k)-\psi(k-1))\sum_{i\in A(\alpha, n, k)\cap E} x_{(\alpha, A(\alpha, n, k))}(i).$$
If $n\leq m$, then 
$$T(n, E)\ \leq\ \psi(m) \ =\ \phi^2(m).$$
Suppose that $n > m$. Write $E = \bigcup_{i=1}^{t_\alpha(E)+1} A_i$ for $A_i\in \mathcal{S}_\alpha$.
By Lemma \ref{boundby6}, 
\begin{equation}\label{re53}\sum_{k=1}^{n}\sum_{i\in A(\alpha, n, k)\cap A_j} x_{(\alpha, A(\alpha, n, k))}(i)\ \le\ 6, \mbox{ with } 1\le j\le t_\alpha(E)+1.\end{equation}
Therefore, if we let 
\begin{equation*}a_k\ :=\ \sum_{i\in A(\alpha, n, k)\cap E} x_{(\alpha, A(\alpha, n, k))}(i)\ \leq\ 1, \mbox{ for }1\leq k\leq n,\end{equation*}
then it follows from \eqref{re53} that 
\begin{equation}\label{e5'}\sum_{k=1}^n a_k\ =\ \sum_{j=1}^{t_\alpha(E)+1}\sum_{k=1}^{n}\sum_{i\in A(\alpha, n, k)\cap A_j} x_{(\alpha, A(\alpha, n, k))}(i)\ \le\ 6(t_\alpha(E)+1).\end{equation}
Due to decreasing $\psi(k)-\psi(k-1)$ for $1\le k\le n$, we have
\begin{align*}
&\sum_{k=t_\alpha(E)+2}^n (\psi(k)-\psi(k-1))a_k\nonumber\\
&\ \le\ \left(\psi(t_\alpha(E)+1)-\psi(t_\alpha(E))\right)\sum_{k=t_\alpha(E)+2}^n a_k\nonumber\\
&\ \le\ \left(\psi(t_\alpha(E)+1)-\psi(t_\alpha(E))\right)\left(6(t_\alpha(E)+1)-\sum_{k=1}^{t_\alpha(E)+1} a_k\right)\mbox{ \quad(by \eqref{e5'})}\nonumber\\
&\ =\ \left(\psi(t_\alpha(E)+1)-\psi(t_\alpha(E))\right)\sum_{k=1}^{t_\alpha(E)+1} (6-a_k)\nonumber \\
&\ \le\ \sum_{k=1}^{t_\alpha(E)+1} (\psi(k)-\psi(k-1))(6-a_k) \mbox{ \quad(due to decreasing $\psi(k)-\psi(k-1)$)}.
\end{align*}
Therefore, 
\begin{align*}T(n, E)&\ =\ \sum_{k=1}^{n}(\psi(k)-\psi(k-1))a_k\\
&\ =\  \sum_{k=1}^{t_\alpha(E)+1}(\psi(k)-\psi(k-1))a_k + \sum_{k=t_\alpha(E) + 2}^{n}(\psi(k)-\psi(k-1))a_k\\
&\ \le\ \sum_{k=1}^{t_\alpha(E)+1}(\psi(k)-\psi(k-1))a_k + \sum_{k=1}^{t_\alpha(E)+1} (\psi(k)-\psi(k-1))(6-a_k)\\
&\ =\ 6\sum_{k=1}^{t_\alpha(E)+1}(\psi(k)-\psi(k-1))\\
&\ =\ 6\psi(t_\alpha(E)+1)\ \le\ 6\psi(m+1) \quad\mbox{ (by Lemma \ref{l10})}.
\end{align*}
This completes our proof. 
\end{proof}

\begin{prop}
The basis $(e_i)_i$ is $\mathcal{S}_{\alpha+1}$-democratic. 
\end{prop}

\begin{proof}
Let $A\in \mathcal{S}_{\alpha+1}$ and $B\in [\mathbb{N}]^{<\infty}$ with $|A|\leq |B|$. Let $m_1$ be the smallest positive integer such that $[m_1, m_1+|A|-1]\in \mathcal{S}_{\alpha+1}$. 
It follows from Lemmas \ref{l1}, \ref{l2}, and \ref{l3} that 
$$\|1_A\|\ \leq\ 6\phi(m_1+1).$$
By Proposition \ref{rp50}, 
$$\|1_B\|\ \geq\ \frac{1}{6}\phi(m_2+1),$$
where $m_2$ is the smallest positive integer such that $[m_2, m_2+|B|-1]\in \mathcal{S}_{\alpha+1}$. Since $|B|\geq |A|$, we know that $m_2\geq m_1$, so $\|1_A\|\leq 36\|1_B\|$. This shows that $(e_i)_i$ is $\mathcal{S}_{\alpha+1}$-democratic. 
\end{proof}

\begin{prop}
The basis $(e_i)_i$ is not $\mathcal{S}_{\alpha+2}$-democratic.
\end{prop}

\begin{proof}
Choose $m\in M_2$ and let $A = [m, s_{(\alpha+2, m)}(1)-1] = \bigcup_{k=1}^{m}A(\alpha+1, m, k)$. Observe that for all $k\in [1,m]$, 
$$m\ \leq\ s_{(\alpha+1, m)}(k-1) \ <\ s_{(\alpha+2, m)}(1).$$
Hence,
\begin{align}\label{e15}\|1_A\|_2&\ =\ \sum_{k=1}^m \phi(s_{(\alpha+1, m)}(k-1))\nonumber\\
&\ \geq\ \sum_{k=1}^m \theta_{(\alpha+1, m)}(s_{(\alpha+1, m)}(k-1))\nonumber\\
&\ =\ \sum_{k=1}^m (\log m + (k-1))\ =\ m\log m + \frac{m(m-1)}{2}.
\end{align}
On the other hand, given a set $B\in \mathcal{S}_{\alpha+1}$ with $|B| = |A|$, it follows from Lemmas \ref{l1}, \ref{l2}, and \ref{l3} that 
\begin{equation}\label{e10}\|1_B\|\ \leq\ 6\phi(m'+1),\end{equation}
where $m'$ is the smallest positive integer such that $[m', m'+|A|-1]\in \mathcal{S}_{\alpha+1}$. 

Let $d\geq \sum_{k=1}^m s_{(\alpha+1, m)}(k-1)$ and write $[d, d+|A|-1]$ as 
$$\bigcup_{k=1}^m \bigcup_{u=1}^{s_{(\alpha+1, m)}(k-1)} \underbrace{\left(d+\sum_{j=1}^{k-1}|A(\alpha+1, m, j)| + \left[\sum_{v = 1}^{u-1}|G^{(k)}_v|,\sum_{v = 1}^{u}|G^{(k)}_v|-1\right]\right)}_{=: A_{k,u}},$$
where 
$$A(\alpha+1, m, k) \ =\ \bigcup_{u=1}^{s_{(\alpha+1, m)}(k-1)}G^{(k)}_u,\mbox{ for }G^{(k)}_u\in \Max(\mathcal{S}_\alpha).$$
Since, for $1\le u\le s_{(\alpha+1, m)}(k-1)$,
$$\min G^{(k)}_u \ =\ s_{(\alpha+1, m)}(k-1) + \sum_{v=1}^{u-1}|G^{(k)}_v|\ \leq\ \min A_{k,u}\mbox{ and }|A_{k,u}| = |G^{(k)}_u|,$$
we know that $A_{k,u}\in \mathcal{S}_\alpha$. It follows that 
$[d, d+|A|-1]$ is the union of $\sum_{k=1}^m s_{(\alpha+1, m)}(k-1)$ sets in $\mathcal{S}_\alpha$; therefore, $d\geq \sum_{k=1}^m s_{(\alpha+1, m)}(k-1)$ implies that $[d, d+|A|-1]\in \mathcal{S}_{\alpha+1}$. The minimality of $m'$ implies that 
\begin{equation}\label{e11}m'\ \leq\ \sum_{k=1}^m s_{(\alpha+1, m)}(k-1).\end{equation}

From \eqref{e10} and \eqref{e11}, we have
\begin{align}\label{e16}
\|1_B\|&\ \leq\ 6\phi\left(\sum_{k=1}^m s_{(\alpha+1, m)}(k-1)+1\right)\ \leq\ 6\phi\left(\sum_{k=1}^m \frac{s_{(\alpha+1, m)}(m-1)}{2^{m-k}}+1\right)\nonumber\\
&\ \leq\ 6\phi(3s_{(\alpha+1, m)}(m-1))\nonumber\\
&\ \leq\ 18\phi(s_{(\alpha+1, m)}(m-1)) \ \leq\ 18\sqrt{2}(\log m + m-1).
\end{align}

We deduce from \eqref{e15} and \eqref{e16} that $\|1_A\|/\|1_B\|\rightarrow\infty$ as $m\rightarrow \infty$, so $(e_i)_i$ is not $\mathcal{S}_{\alpha+2}$-democratic. 
\end{proof}

\subsection{$\boldsymbol{\mathcal{S}_{\alpha}}$-unconditional but not $\boldsymbol{\mathcal{S}_{\alpha+1}}$-unconditional}

\begin{prop}
The basis $(e_i)_i$ is $\mathcal{S}_\alpha$-unconditional. 
\end{prop}
\begin{proof}
Let $x = \sum_{i}x_ie_i$ with $\|x\| = 1$. Due to $\|\cdot\|_1$, $|x_i|\leq 1$ for all $i\in \mathbb{N}$. Pick $E\in \mathcal{S}_\alpha$. It suffices to show that for every $m\in M_1$ and $i_0\ge 1$, 
$$\left|\sum_{k=1}^m(\phi(k)-\phi(k-1))\sum_{\substack{i\in A(\alpha, m, k)\cap E\\ i\le i_0}}x_{(\alpha, A(\alpha, m, k))}(i)x_i\right| \ \le\ 6.$$
Indeed, by Lemma \ref{boundby6},
$$\sum_{k=1}^m\left|\sum_{\substack{i\in A(\alpha, m, k)\cap E\\ i\le i_0}}x_{(\alpha, A(\alpha, m, k))}(i)x_i\right|\ \le\ \sum_{k=1}^m\sum_{\substack{i\in A(\alpha, m, k)\cap E}}x_{(\alpha, A(\alpha, m, k))}(i)\ \le\ 6.$$
It follows from decreasing $\phi(k)-\phi(k-1)$ that
\begin{align*}&\left|\sum_{k=1}^m (\phi(k)-\phi(k-1))\sum_{\substack{i\in A(\alpha, m, k)\cap E\\ i\le i_0}}x_{(\alpha, A(\alpha, m, k))}(i)x_i\right|\\
&\ \le\ \sum_{k=1}^m (\phi(1)-\phi(0))\left|\sum_{\substack{i\in A(\alpha, m, k)\cap E\\ i\le i_0}}x_{(\alpha, A(\alpha, m, k))}(i)x_i\right|\\
&\ =\ \sum_{k=1}^m\left|\sum_{\substack{i\in A(\alpha, m, k)\cap E\\ i\le i_0}}x_{(\alpha, A(\alpha, m, k))}(i)x_i\right|\ \le\ 6.
\end{align*}
This completes our proof. 
\end{proof}

For the next step, we need the following lemma.
\begin{lem}\label{bounduncon}
For each integer $m\in M_1$ and $q\in [1, m]$, it holds that
$$\frac{\phi(q)}{q}\sum_{i=1}^q\frac{1}{\phi(i)}\ \le\ \log^{1/4} m +3.$$
\end{lem}
\begin{proof}
If $q\le \log m$, then 
$$\frac{\phi(q)}{q}\sum_{i=1}^q\frac{1}{\phi(i)}\ \leq\ \phi(q)\ \leq\ \log^{1/4} m.$$
If $q > \log m$, we have, by \eqref{re31} and \eqref{re32}, 
\begin{align*}
\frac{\phi(q)}{q}\sum_{i=1}^q \frac{1}{\phi(i)}&\ =\ \frac{\phi(q)}{q}\left(\sum_{i = 1}^{\lceil \log m\rceil}\frac{1}{\phi(i)} + \sum_{i=\lceil \log m\rceil+1}^q \frac{1}{\phi(i)}\right)\\
&\ \leq\ \frac{1}{q^{3/4}}\left(\log m + 1 + \int_{\log m}^q\frac{dx}{x^{1/4}}\right)\\
&\ \leq\ \frac{1}{q^{3/4}}\left(\log m + 1 + \frac{4}{3}q^{3/4}\right)\\
&\ \leq\ \log^{1/4} m + 3,
\end{align*}
as desired. 
\end{proof}

\begin{prop}\label{whyS1}
The basis $(e_i)_i$ is not $\mathcal{S}_{\alpha+1}$-unconditional.
\end{prop}

\begin{proof}
Choose $m\in M_1$ and define
$$x\ =\ \sum_{k=1}^m \frac{1}{\phi(k)}\sum_{i\in A(\alpha, m, k)}e_i\quad \mbox{ and }\quad y\ =\ \sum_{k=1}^m \frac{1}{\phi(k)}\sum_{i\in A(\alpha, m, k)}(-1)^i e_i.$$
For sufficiently large $m$, we have
\begin{align*}
\|x\|&\ \geq\ \|x\|_4\ =\ \sum_{k=1}^m \frac{\phi(k)-\phi(k-1)}{\phi(k)} \ \geq\ \sum_{k=\lceil \log m \rceil}^m \frac{\sqrt[4]{k}-\sqrt[4]{k-1}}{\sqrt[4]{k}}\\
&\ \geq\ \frac{1}{4}\sum_{k=\lceil \log m \rceil}^m \frac{1}{k}\ \geq\ \frac{1}{4}\int_{2\log m}^m \frac{dx}{x}\ \geq\ \frac{1}{5}\log m.
\end{align*}

Let us bound $\|y\|$ from above. Due to \eqref{e17}, $\|y\|_2 = 0$. Furthermore, due to the alternating sum and Property (P2) in Section \ref{Schreierfam}, 
$$\|y\|_4\ \leq\ (\phi(1)-\phi(0))x_{(\alpha, A(\alpha, m, 1))}(m)\frac{1}{\phi(1)}\ \leq\ 1.$$
Next, we have
\begin{align*}
\|y\|_3 &\ =\ \left(\sum_{k=1}^{m}\frac{\psi(k)-\psi(k-1)}{\psi(k)}  \right)^{1/2}\ \leq\ \left(\log m + \sum_{k=\lceil \log m\rceil}^m \frac{\sqrt{k}-\sqrt{k-1}}{\sqrt{k}}\right)^{1/2}\\
&\ \leq\ \left(\log m + \sum_{k=2}^m \frac{1}{k}\right)^{1/2}\ \leq\ \sqrt{2}\log^{1/2} m.
\end{align*}

Finally, we find an upper bound for $\|y\|_1$. For $s\le A_1 < A_2 < \cdots < A_s$ in $\Max(\mathcal{S}_\alpha)$ and  
an increasing map $\pi: \bigcup_{j=1}^s A_j\rightarrow \mathbb{N}$ with $\pi(\bigcup_{j=1}^s A_j)\in \mathcal{S}_1$, define
$$
T\left(\bigcup_{j=1}^s A_j, \pi\right) \ :=\ \frac{\phi(s)}{s}\sum_{j=1}^s\sum_{i\in A_j} x_{(\alpha, A_j)}(i)|y_{\pi(i)}|.
$$
Let $A = \bigcup_{j=1}^s A_j$.

Case 1: $\alpha = 0$. Then $A_j$'s are singletons and $|A| = s$. We have
$$
T\left(A, \pi\right) \ =\ \frac{\phi(s)}{s}\sum_{i\in A} |y_{\pi(i)}|\ \le\ \frac{\phi(s)}{s}\sum_{i=1}^{\min\{|A|, m\}}\frac{1}{\phi(i)}.
$$
If $|A|\le \log m$, then
$$T(A, \pi)\ \le\ \frac{\phi(|A|)}{|A|}|A|\ \le\ \log^{1/4} m.$$
If $|A|\in (\log m, m]$, by Lemma \ref{bounduncon},
$$
    T(A, \pi)\ \le\ \frac{\phi(|A|)}{|A|}\sum_{i=1}^{|A|}\frac{1}{\phi(i)}\ \le\ \log^{1/4} m + 3. 
$$
If $|A| > m$, since $\phi(x)/x$ is decreasing, and by Lemma \ref{bounduncon}, it follows that
$$T(A, \pi)\ \le\ \frac{\phi(|A|)}{|A|}\sum_{i=1}^m \frac{1}{\phi(i)}\ \le\ \frac{\phi(m)}{m}\sum_{i=1}^m \frac{1}{\phi(i)}\ \le\ \log^{1/4} m + 3.$$
Hence, $T\left(A, \pi\right)\le \log^{1/4} m + 3$.

Case 2: $\alpha\ge 1$. 
Without loss of generality, we assume that  $$\min \pi(A)\in [m, s_{(\alpha+1, m)}(1)-1].$$ 
Since $(|y(i)|)_{i\ge m}$ is decreasing, in finding an upper bound for $\|y\|_1$, we can further assume that $\pi(A)$ is an interval. Then $\pi(A)\in \mathcal{S}_1$ implies that there exists $p\in [1,m]$ such that 
\begin{equation}\label{gc1}\pi(A)\ \subset\ A(\alpha, m, p)\cup A(\alpha, m, p+1).\end{equation}
It follows from \eqref{gc1} that
\begin{align*}
T(A, \pi)&\ =\ \frac{\phi(s)}{s}\sum_{j=1}^s\left(\sum_{i\in A_j\cap \pi^{-1}(A(\alpha, m, p))}\frac{x_{(\alpha, A_j)}(i)}{\phi(p)} + \sum_{i\in A_j\cap \pi^{-1}(A(\alpha, m, p+1))}\frac{x_{(\alpha, A_j)}(i)}{\phi(p+1)}\right)\\
&\ \le\ \frac{\phi(s)}{s}\frac{s}{\phi(p)}\ =\ \frac{\phi(s)}{\phi(p)}.
\end{align*}
If $s \le p$, $T(A, \pi)\le 1$. Assume that $s \ge p+1$. For $\beta < \omega_1$, define the function $\Gamma_{\beta}: \mathbb{N}\rightarrow\mathbb{N}$ as $\Gamma_\beta(i) := s_{(\beta, i)}(1)$. 
By \eqref{gc1}, 
\begin{equation}\label{re20}|A|\ =\ |\pi(A)|\ \le\ 2|A(\alpha, m, p+1)|\ \le\ 2\Gamma^{(p+1)}_{\alpha}(m),\end{equation}
where $f^{(k)}$ is the $k$-time composition of a function $f$.
On the other hand,
\begin{equation}\label{re21}|A|\ \ge\ |[s, s_{(\alpha+1, s)}(1)-1]|\ =\ \Gamma_{\alpha+1}(s)-s.\end{equation}
We deduce from \eqref{re20} and \eqref{re21} that 
\begin{equation}\label{eee32}\Gamma_{\alpha+1}(s)-s\ \le\ 2\Gamma^{(p+1)}_\alpha(m).\end{equation}
Since $\alpha\ge 1$ and $s\ge 1$, 
$$\frac{1}{2}\Gamma_{\alpha+1}(s)\ \ge\ \frac{1}{2}\Gamma_2(s)\ \ge\ s2^{s-1}\ \ge\ s\ \Longrightarrow\ \Gamma_{\alpha+1}(s) - s \ \ge\ \frac{1}{2}\Gamma_{\alpha+1}(s).$$
Then \eqref{eee32} implies that
\begin{equation}\label{eee33}\Gamma^{(s)}_{\alpha}(s)\ \le\ 4\Gamma^{(p+1)}_{\alpha}(m).\end{equation}
We claim that $s\le p+ \lceil \log_2 m\rceil + 2$. Suppose, for a contradiction, that $s \ge p+ \lceil \log_2 m\rceil + 3$. 
Then 
\begin{align*}\Gamma^{(s)}_{\alpha}(s)&\ =\ \Gamma^{(2)}_\alpha(\Gamma^{(s-\lceil \log_2 m\rceil-2)}_\alpha (\Gamma^{\lceil \log_2 m\rceil}_\alpha (s)))\\
&\ \ge\ \Gamma^{(2)}_1(\Gamma^{(p+1)}_\alpha (\Gamma^{\lceil \log_2 m\rceil}_1 (s)))\\
&\ =\ \Gamma^{(2)}_1(\Gamma^{(p+1)}_\alpha (s2^{\lceil \log_2 m\rceil}))\\
&\ >\ 4\Gamma^{(p+1)}_\alpha (m)\quad \mbox{ (because }s\ge p+1\ge 2\mbox{)},\\
\end{align*}
which contradicts \eqref{eee33}. It follows that
\begin{align*}
T(A, \pi)&\ \le\ \frac{\phi(p+ \lceil \log_2 m\rceil + 2)}{\phi(p)}\\
&\ \le\ \frac{\phi(p+3)}{\phi(p)} + \frac{\phi(\log_2 m)}{\phi(p)}\\
&\ \le\ \frac{p+3}{p} + \log_2^{1/4} m\ <\ 2\log^{1/4} m + 4.
\end{align*}

We have shown that for sufficiently large $m$, $\|y\|\ \leq\ \sqrt{2}\log^{1/2} m$; meanwhile, $\|x\|\geq (\log m)/5$. Hence, $(e_i)_i$ is not $\mathcal{S}_{\alpha+1}$-unconditional. 
\end{proof}

\subsection{Quasi-greedy}
Since the semi-norms $\|\cdot\|_1, \|\cdot\|_2$, and $\|\cdot\|_3$ are unconditional. To show that $(e_i)$ is quasi-greedy, we need only to prove the following.

\begin{prop}\label{QGp}
Let $x = \sum_{i=1}^\infty x_ie_i\in X$, with $\|x\| = 1$. For all $\varepsilon \in (0,1]$, $m\in M_1$, and $i_0\ge 1$, we have
$$\left|\sum_{k=1}^m(\phi(k)-\phi(k-1)))\sum_{\substack{i\in L\cap A(\alpha, m, k)\\ i\le i_0}}x_{(\alpha, A(\alpha, m, k))}(i)x_i\right|\ \leq\ 3,$$
where $L = \{i: |x_i|\geq \varepsilon\}$.
\end{prop}

\begin{proof}
We have
\begin{align*}
&\left|\sum_{k=1}^m(\phi(k)-\phi(k-1))\sum_{i\in L\cap A(\alpha, m, k), i\le i_0}x_{(\alpha, A(\alpha, m, k))}(i)x_i\right|\\
\ \leq\ &\left|\sum_{k=1}^m(\phi(k)-\phi(k-1))\sum_{i\in A(\alpha, m, k), i\le i_0}x_{(\alpha, A(\alpha, m, k))}(i)x_i\right|\\
&\qquad +\left|\sum_{k=1}^m(\phi(k)-\phi(k-1))\sum_{i\in A(\alpha, m, k)\backslash L, i\le i_0}x_{(\alpha, A(\alpha, m, k))}(i)x_i\right|\\
\ \leq\ &\|x\| + \varepsilon\phi(m).
\end{align*}

Case 1: $m\leq \phi^{-1}(1/\varepsilon)$. We deduce that 
$$\left|\sum_{k=1}^m(\phi(k)-\phi(k-1))\sum_{i\in L\cap A(\alpha, m, k), i\le i_0}x_{(\alpha, A(\alpha, m, k))}(i)x_i\right|\ \leq\ 2.$$

Case 2: $\phi^{-1}(1/\varepsilon) < m$. Fix $j_0\in [1, m-1]$ such that $j_0 \leq \phi^{-1}(1/\varepsilon) < j_0+1$. It follows from Case 1 and H\"{o}lder's Inequality that
\begin{align*}
&\left|\sum_{k=1}^m(\phi(k)-\phi(k-1))\sum_{\substack{i\in L\cap A(\alpha, m, k)\\ i\le i_0}}x_{(\alpha, A(\alpha, m, k))}(i)x_i\right|\\
\ \leq\ &2 + \left|\sum_{k=j_0+1}^m(\phi(k)-\phi(k-1))\sum_{\substack{i\in L\cap A(\alpha, m, k)\\ i\le i_0}}x_{(\alpha, A(\alpha, m, k))}(i)x_i\right|\\
\ \le\ &2+ \sum_{\substack{j_0+1\leq k\leq m\\ i\in L\cap A(\alpha, m, k)\\ i\le i_0}}\frac{(x_{(\alpha, A(\alpha, m, k))}(i))^{1/3}(\phi(k)-\phi(k-1))}{(\phi^2(k)-\phi^2(k-1))^{2/3}} \\
&\qquad\qquad\qquad\qquad \cdot  (\phi^2(k)-\phi^2(k-1))^{2/3}(x_{(\alpha, A(\alpha, m, k))}(i))^{2/3}\left|x_i\right|\\
\ \leq\ &2+ \left(\sum_{\substack{j_0+1\leq k\leq m\\ i\in L\cap A(\alpha, m, k)\\ i\le i_0}}\frac{x_{(\alpha, A(\alpha, m, k))}(i)(\phi(k)-\phi(k-1))^3}{(\phi^2(k)-\phi^2(k-1))^{2}}\right)^{1/3}\\
&\qquad\qquad\qquad\qquad \cdot \left(\sum_{\substack{j_0+1\leq k\leq m\\ i\in L\cap A(\alpha, m, k)\\ i\le i_0}} (\phi^2(k)-\phi^2(k-1))x_{(\alpha, A(\alpha, m, k))}(i)|x_i|^{3/2}\right)^{2/3}.
\end{align*}
We estimate the first factor as follows:
\begin{align*}
&\left(\sum_{\substack{j_0+1\leq k\leq m\\ i\in L\cap A(\alpha, m, k)\\ i\le i_0}}\frac{x_{(\alpha, A(\alpha, m, k))}(i)(\phi(k)-\phi(k-1))^3}{(\phi^2(k)-\phi^2(k-1))^{2}}\right)^{1/3}\\
\ \leq\ &\left(\sum_{k=j_0+1}^m\frac{\phi(k)-\phi(k-1)}{(\phi(k)+\phi(k-1))^{2}}\right)^{1/3}\\
\ \leq\ &\left(\frac{1}{4}\sum_{k=j_0+1}^\infty\frac{\phi(k)-\phi(k-1)}{\phi(k)\phi(k-1)}\right)^{1/3}\\
\ =\ &\left(\frac{1}{4}\sum_{k=j_0+1}^\infty\left(\frac{1}{\phi(k-1)}-\frac{1}{\phi(k)}\right)\right)^{1/3}\\
\ =\ &\frac{1}{4^{1/3}}\frac{1}{\phi^{1/3}(j_0)}\\
\ =\ &\frac{1}{4^{1/3}}\frac{\phi^{1/3}(j_0+1)}{\phi^{1/3}(j_0)}\frac{1}{\phi^{1/3}(j_0+1)}\ \leq\ \varepsilon^{1/3}.
\end{align*}
To estimate the second factor, we observe that for each $A(\alpha, m, k)$, 
$$\sum_{\substack{i\in L\cap A(\alpha, m, k)\\ i\le i_0}}x_{(\alpha, A(\alpha, m, k))}(i)|x_i|^{3/2}\ \leq\ \varepsilon^{-1/2}\sum_{\substack{i\in L\cap A(\alpha, m, k)\\ i\le i_0}}x_{(\alpha, A(\alpha, m, k))}(i)|x_i|^2;$$
hence,
\begin{align*}
&\left(\sum_{k=j_0+1}^m (\phi^2(k)-\phi^2(k-1))\sum_{\substack{i\in L\cap A(\alpha, m, k)\\ i\le i_0}}x_{(\alpha, A(\alpha, m, k))}(i)|x_i|^{3/2}\right)^{2/3}\\
\ \leq\ & \varepsilon^{-1/3}\left(\sum_{k=j_0+1}^m (\phi^2(k)-\phi^2(k-1))\sum_{\substack{i\in L\cap A(\alpha, m, k)\\ i\le i_0}}x_{(\alpha, A(\alpha, m, k))}(i)|x_i|^2\right)^{2/3}\\
\ \leq\ & \varepsilon^{-1/3}(\|x\|^2)^{2/3}\ =\ \varepsilon^{-1/3}.
\end{align*}
Combining our estimates, we obtain
$$\left|\sum_{k=1}^m(\phi(k)-\phi(k-1))\sum_{\substack{i\in L\cap A(\alpha, m, k)\\ i\le i_0}}x_{(\alpha, A(\alpha, m, k))}(i)x_i\right|\ \leq\ 2+\varepsilon^{1/3}\varepsilon^{-1/3}\ =\ 3,$$
which finishes the proof. 
\end{proof}

\section{Construction of an $(\alpha,\beta)$-quasi-greedy basis for $\beta \le \alpha$ and $(\alpha, \beta)\neq (0,0)$}
We first construct an example of an $(\alpha,\alpha)$-quasi-greedy basis for each $\alpha \ge 1$ then an $(\alpha,\beta)$-quasi-greedy basis for $0\le \beta < \alpha$.
\subsection{An $\boldsymbol{(\alpha, \alpha)}$-quasi-greedy basis for $\boldsymbol{\alpha \ge 1}$}\label{ex(a,a)}
For $i\in \mathbb{N}$, let $F_i := [s_{(\alpha, 1)}(i-1), s_{(\alpha, 1)}(i)-1]$ (recall the definition of $s_{(\alpha, m)}(i)$ from Section \ref{(alpha, alpha+1)}), and thus, $\bigcup_{i=1}^{\infty}F_i = \mathbb{N}$, and $F_1 < F_2 < F_3 < \cdots$ are in $\Max(\mathcal{S}_\alpha)$. 
Given $(x_i)_{i=1}^{\infty}\in c_{00}$, define 
\begin{align*}
    \|(x_i)_i\|_0&\ =\ \max_{i}|x_i|,\\
    \|(x_i)_i\|_1 &\ =\ \left(\sum_{j=1}^{\infty} \sum_{i\in F_j}x_{(\alpha, F_j)}(i)x_i^2\right)^{1/2},\\
    \|(x_i)_i\|_2 &\ =\ \sup_{N, i_0\in \mathbb{N}}\sum_{j=N}^{2N-1}\frac{1}{\sqrt{j-N+1}}\left|\sum_{i\in F_j, i\le i_0}x_{(\alpha, F_j)}(i)x_i\right|.
\end{align*}
Let $X$ be the completion of $c_{00}$ with respect to the norm $\|\cdot\| = \max\{\|\cdot\|_{0}, \|\cdot\|_{1}, \|\cdot\|_2\}$. Clearly, the canonical basis $(e_i)_i$ is a normalized Schauder basis of $X$. 

\begin{exa}\normalfont
In the case $\alpha = 1$, we have $$F_k \ =\ \{2^{k-1}, \ldots, 2^k-1\}, k\in \mathbb{N},$$
and for $(x_i)_{i=1}^{\infty}\in c_{00}$,  
\begin{align*}
    \|(x_i)_i\|_0&\ =\ \max_{i}|x_i|,\\
    \|(x_i)_i\|_1 &\ =\ \left(\sum_{j=1}^{\infty} \frac{1}{2^{j-1}}\sum_{i\in F_j}x_i^2\right)^{1/2},\\
    \|(x_i)_i\|_2 &\ =\ \sup_{N, i_0\in \mathbb{N}}\sum_{j=N}^{2N-1}\frac{1}{2^{j-1}\sqrt{j-N+1}}\left|\sum_{i\in F_j, i\le i_0}x_i\right|.
\end{align*}    
\end{exa}

\begin{prop}\label{pe10}
The basis $(e_i)_i$ is $\mathcal{S}_\alpha$-democratic but not $\mathcal{S}_{\alpha+1}$-democratic.    
\end{prop}
\begin{proof}
For any $A\in \mathcal{S}_\alpha$, by Lemma \ref{boundby6}, $\|1_A\|\le 6$. Hence, if $B\in [\mathbb{N}]^{<\infty}$ with $|B|\ge |A|$, we have $\|1_A\|\le 6\|1_B\|_0\le 6\|1_B\|$, and thus, $(e_i)_i$ is $\mathcal{S}_\alpha$-democratic. 

Next, we show that $(e_i)_i$ is not $\mathcal{S}_{\alpha+1}$-democratic. Let $E_N = \cup_{j=N}^{2N-1} F_j$, which is in $\mathcal{S}_{\alpha+1}$ because each $F_j$ is in $\mathcal{S}_\alpha$ and $\min F_N \ge N$. We have 
$$\|1_{E_N}\|\ \ge\ \|1_{E_N}\|_2\ =\ \sum_{j=N}^{2N-1}\frac{1}{\sqrt{j-N+1}}\ =\ \sum_{j=1}^N \frac{1}{\sqrt{j}}\ \ge\ \sqrt{N}.$$
On the other hand, if $\widetilde{E}_N$ is in $\mathcal{S}_\alpha$ and $|\widetilde{E}_N| = |E_N|$, then it follows from the first part of the proof that $\|1_{\widetilde{E}_N}\|\le 6$. Since $\|1_{E_N}\|/\|1_{\widetilde{E}_N}\|\rightarrow\infty$ as $N\rightarrow\infty$, $(e_i)_i$ is not $\mathcal{S}_{\alpha+1}$-democratic. 
\end{proof}

\begin{prop}\label{rp11}
The basis $(e_i)_i$ is $\mathcal{S}_\alpha$-unconditional but not $\mathcal{S}_{\alpha+1}$-unconditional.    
\end{prop}
\begin{proof}
The basis $(e_i)_i$ is unconditional with respect to the norms $\|\cdot\|_0$ and $\|\cdot\|_1$. Therefore, in order to show $\mathcal{S}_\alpha$-unconditionality, it suffices to show for $x\in X$, with $\|x\| = 1$, and $F\in \mathcal{S}_\alpha$, that $\|P_F(x)\|_2\le 6$. 

Since $|x_i|\le 1$ for all $i\in \mathbb{N}$, Lemma \ref{boundby6} yields, for $N, i_0\in \mathbb{N}$, that
$$\sum_{j=N}^{2N-1}\frac{1}{\sqrt{j-N+1}}\left|\sum_{i\in F_j\cap F, i\le i_0}x_{(\alpha, F_j)}(i)x_i\right|\ \le\ \sum_{j=N}^{2N-1}\sum_{i\in F_j\cap F}x_{(\alpha, F_j)}(i)\ \le\ 6,$$
which proves our claim. 

To see that $(e_i)_i$ is not $\mathcal{S}_{\alpha+1}$-unconditional, we define  
    $$x \ =\ x_N\ =\  \sum_{j=N}^{2N-1} \sum_{i\in F_j}\frac{(-1)^i}{\sqrt{j-N+1}}e_{i}$$
and
    $$y \ =\ y_N\ =\ \sum_{j=N}^{2N-1} \sum_{i\in F_j}\frac{1}{\sqrt{j-N+1}}e_{i}.$$
It is easy to see that $\|x\|_0 = \|y\|_0 = 1$, $\|x\|_1 = \|y\|_1 = (\sum_{j=1}^N 1/j)^{1/2}$, and by the alternating sum criteria,
\begin{equation}\label{re40}\|x\|_2\ \le\ \sum_{j=N}^{2N-1} \frac{x_{(\alpha, F_j)}(\min F_j)}{j-N+1}\ =\ \sum_{j=1}^N\frac{x_{(\alpha, F_{j+N-1})}(\min F_{j+N-1})}{j}.\end{equation}
By Properties (P1) and (P2) in Section \ref{Schreierfam}, for $j\ge 1$, 
\begin{equation}\label{re41}x_{(\alpha, F_{j+1})}(\min F_{j+1})\ \le\ x_{(\alpha, F_j)}(\max F_j)\ \le\ \frac{1}{|F_j|}.\end{equation}
We deduce from \eqref{re40} and \eqref{re41} that
\begin{align*}\|x\|_2&\ \le\ 1 + \sum_{j=2}^N\frac{x_{(\alpha, F_{j+N-1})}(\min F_{j+N-1})}{j}\\
&\ \le\ 1 + \sum_{j=2}^N \frac{1}{|F_{j+N-2}|j}\\
&\ \le\  1 + \sum_{j=2}^\infty \frac{1}{2^{j-2}j}\ < \ 3.
\end{align*}
Hence, for sufficiently large $N$, $\|x\| = \|x\|_1 = (\sum_{j=1}^N 1/j)^{1/2}$.

On the other hand,
\begin{align*}
    \|y\|\ \ge\ \|y\|_2\ \ge\ \sum_{j=N}^{2N-1}\frac{1}{j-N+1}\ =\ \sum_{j=1}^{N}\frac{1}{j}. 
\end{align*}
Therefore, $\|y_N\|/\|x_N\|\rightarrow\infty$ as $N\rightarrow\infty$, and thus, $(e_i)_i$ is not $\mathcal{S}_{\alpha+1}$-unconditional. 
\end{proof}

\begin{prop}\label{QGgen}
The basis $(e_i)_i$ is quasi-greedy.    
\end{prop}

\begin{proof}
It is clear that $(e_i)_i$ is quasi-greedy as basis of the completion of $c_{00}$ with respect to the norms $\|\cdot\|_0$ and $\|\cdot\|_1$. It, therefore, suffices to prove that for $(x_i)_i\in c_{00}$, with $\|(x_i)_i\| = 1$, it follows that 
$$\sum_{j=N}^{2N-1}\frac{1}{\sqrt{j-N+1}}\left|\sum_{i\in \Lambda_j, i\le i_0}x_{(\alpha, F_j)}(i)x_i\right|\ \leq\ 3+\sqrt{2},$$
for all $\varepsilon > 0$, for all $N, i_0\in \mathbb{N}$, and $\Lambda_j = \{i\in F_j: |x_i|>\varepsilon\}$. Since $\max_i |x_i| \leq 1$, we can assume without loss of generality, that $0 < \varepsilon < 1$. Set $L = \lfloor \varepsilon^{-2}\rfloor$ to have $1/2\leq \varepsilon^2 L \leq 1$. We distinguish between two cases. 

For $M\leq \min\{2N-1, N+L-1\}$, we have 
\begin{align}\label{ar1}
&\sum_{j=N}^{M}\frac{1}{\sqrt{j-N+1}}\left|\sum_{i\in  \Lambda_j, i\le i_0}x_{(\alpha, F_j)}(i)x_i\right|\nonumber\\
&\ \leq\ \sum_{j=N}^{M}\frac{1}{\sqrt{j-N+1}}\left|\sum_{i\in F_j, i\le i_0}x_{(\alpha, F_j)}(i)x_i\right| + \sum_{j=N}^{M}\frac{1}{\sqrt{j-N+1}}\left|\sum_{\substack{i\in F_j, i\le i_0\\ |x_i|\leq \varepsilon}}x_{(\alpha, F_j)}(i)x_i\right|\nonumber\\
&\ \leq\ \|(x_i)_i\|+\varepsilon \sum_{j=N}^{M}\frac{1}{\sqrt{j-N+1}}\nonumber\\
&\ =\ 1+ \varepsilon \sum_{j=1}^{M-N+1}\frac{1}{\sqrt{j}}\ \leq\ 1+ 2\varepsilon \sqrt{M-N+1}\ \leq\ 1+2\varepsilon \sqrt{L}\ \leq 3. 
\end{align}

Case 1: $N\le L$, \eqref{ar1} gives
$$\sum_{j=N}^{2N-1}\frac{1}{\sqrt{j-N+1}}\left|\sum_{i\in  \Lambda_j, i\le i_0}x_{(\alpha, F_j)}(i)x_i\right|\ \le\ 3.$$

Case 2: $N > L$, we have
\begin{align*}
&\sum_{j=N}^{2N-1}\frac{1}{\sqrt{j-N+1}}\left|\sum_{i\in  \Lambda_j, i\le i_0}x_{(\alpha, F_j)}(i) x_i\right|\\
&\ =\ \sum_{j=N}^{N+L-1}\frac{1}{\sqrt{j-N+1}}\left|\sum_{i\in  \Lambda_j, i\le i_0}x_{(\alpha, F_j)}(i)x_i\right| + \sum_{j=N+L}^{2N-1}\frac{1}{\sqrt{j-N+1}}\left|\sum_{i\in  \Lambda_j, i\le i_0}x_{(\alpha, F_j)}(i)x_i\right|\\
&\ \leq\ 3 + \sum_{\substack{N+L\leq j\leq 2N-1\\i\in F_j, i\le i_0\\ |x_i|>\varepsilon}}\left|\frac{x_{(\alpha, F_j)}(i)}{\sqrt{j-N+1}}x_i\right|\quad \mbox{ (by \eqref{ar1})}.
\end{align*}
Applying H\"older's Inequality to the second term yields
\begin{align*}
&\sum_{\substack{N+L\leq j\leq 2N-1\\i\in F_j, i\le i_0\\ |x_i|>\varepsilon}}\left|\frac{x_{(\alpha, F_j)}(i)}{\sqrt{j-N+1}}x_i\right|\\
&\ =\ \sum_{\substack{N+L\leq j\leq 2N-1\\i\in F_j, i\le i_0\\ |x_i|>\varepsilon}} \left|\left(x_{(\alpha, F_j)}(i)\right)^{1/3}\frac{1}{\sqrt{j-N+1}}\cdot \left(x_{(\alpha, F_j)}(i)\right)^{2/3}x_i\right|\\
&\ \leq\ \left(\sum_{\substack{N+L\leq j\leq 2N-1\\i\in F_j, i\le i_0\\ |x_i|>\varepsilon}}\frac{x_{(\alpha, F_j)}(i)}{(j-N+1)^{3/2}}\right)^{1/3} \left(\sum_{N+L\leq j\leq 2N-1}\sum_{\substack{i\in F_j, i\le i_0\\ |x_i|>\varepsilon}}x_{(\alpha, F_j)}(i)|x_i|^{3/2}\right)^{2/3}\\
&\ \leq\ \left(\sum_{j=L+1}^{\infty}\frac{1}{j^{3/2}}\right)^{1/3} \left(\varepsilon^{-1/2}\sum_{N+L\leq j\leq 2N-1}\sum_{\substack{i\in F_j, i\le i_0\\ |x_i|>\varepsilon}}x_{(\alpha, F_j)}(i)x_i^{2}\right)^{2/3}\\
&\ \leq\ 2^{1/3}L^{-1/6}\varepsilon^{-1/3} \ \leq \ \sqrt{2}.
\end{align*}
This completes our proof. 
\end{proof}    

\subsection{An $\boldsymbol{(\alpha, \beta)}$-quasi-greedy basis for $\boldsymbol{\beta < \alpha}$}
We slightly modify our $(\alpha, \alpha)$-quasi-greedy basis. Choose two sequences of natural numbers $(m_i)_{i=1}^\infty$ and $(n_i)_{i=1}^\infty$ such that 
$$m_i \ <\ n_i\  <\ 2n_i-1 \ <\ m_{i+1}\quad\mbox{ and }\quad s_{(\beta+1, \min F_{m_i})}(1)\ <\ \min F_{n_i}.$$
For each $i\in \mathbb{N}$, choose 
$$A_i\ =\ [\min F_{m_i}, s_{(\beta+1, \min F_{m_i})}(1)-1] = [s_{(\beta+1,\min F_{m_i})}(0), s_{(\beta+1, \min F_{m_i})}(1)-1],$$ which is an element of $\Max(\mathcal{S}_{\beta+1})$.

For $x = (x_i)_i\in c_{00}$, we define the semi-norm
$$\|x\|_\beta\ =\ \sup_{j}\left(\min A_j\sum_{i\in A_j}x_{(\beta+1, A_j)}(i)|x_i|\right).$$
Let $Y$ be the completion of $c_{00}$ with respect to the following norm:
$$\|x\|_{(\alpha, \beta)}\ :=\ \max\{\|x\|_{(\alpha, \alpha)}, \|x\|_\beta\},$$
where $\|x\|_{(\alpha, \alpha)}$ is the norm defined in Subsection \ref{ex(a,a)}. Clearly, 
the canonical basis $(e_i)_i$ is still quasi-greedy. 

\begin{prop}
The basis $(e_i)_i$ is $\mathcal{S}_\beta$-democratic but not $\mathcal{S}_{\beta+1}$-democratic.
\end{prop}
\begin{proof}
By \eqref{er12}, there exists $N$ such that $\{E\in \mathcal{S}_\beta: N < E\}\subset \mathcal{S}_\alpha$. Let $A\in \mathcal{S}_\beta$. Write $A = A_{\le N}\cup A_{> N}$, where $A_{\le N} = \{i\in A: i\le N\}$ and $A_{> N} = \{i\in A: i > N\}$. By Lemma \ref{boundby6},
$$\begin{cases}
\|1_A\|_\beta\ \le\ 6,\\
\|1_A\|_{(\alpha, \alpha)}\ \le\ \|1_{A_{\le N}}\|_{(\alpha, \alpha)} + \|1_{A_{>N}}\|_{(\alpha, \alpha)}\ \le\ N + 6.\end{cases}$$
Hence, $\|1_A\|_{(\alpha, \beta)}\le N+6$, which implies that $(e_i)_i$ is $\mathcal{S}_\beta$-democratic.

Choose $B_i\in \mathcal{S}_\alpha$ so that 
$B_i\subset \cup_{j=1}^{\infty} F_{n_j}$ and $|A_i|\le |B_i|$. From the proof of Proposition \ref{pe10}, we know that $\|1_{B_i}\|_{(\alpha, \beta)}\ =\ \|1_{B_i}\|_{(\alpha, \alpha)}\le 6$. However, 
$$\|1_{A_i}\|_{(\alpha, \beta)}\ \ge\ \|1_{A_i}\|_{\beta} \ =\ \min A_i\ =\ \min F_{m_i}.$$
Hence, $\|1_{A_i}\|/\|1_{B_i}\|\rightarrow \infty$ as $i\rightarrow\infty$, and thus, $(e_i)_i$ is not $\mathcal{S}_{\beta+1}$-democratic. 
\end{proof}

\begin{prop}
The basis $(e_i)_i$ is $\mathcal{S}_{\alpha}$-unconditional but not $\mathcal{S}_{\alpha+1}$-unconditional.   
\end{prop}

\begin{proof} Thanks to Proposition \ref{rp11} and the unconditional $\|\cdot\|_\beta$, it suffices to show that $(e_i)_i$ is not $\mathcal{S}_{\alpha+1}$-unconditional. For $N\in \mathbb{N}$, let
\begin{align*}
    x &\ =\ x_N\ =\ \sum_{j=n_N}^{2n_N-1}\sum_{i\in F_{j}}\frac{(-1)^i}{\sqrt{j-n_N+1}}e_{i}\mbox{ and }\\
    y &\ =\ y_N\ =\ \sum_{j=n_N}^{2n_N-1}\sum_{i\in F_{j}}\frac{1}{\sqrt{j-n_N+1}}e_{i}.
\end{align*}
Since $\supp(x) = \supp(y) \subset \mathbb{N}\backslash\left(\cup_{j=1}^\infty F_{m_j}\right)$, 
$\|x\|_{(\alpha, \beta)} = \|x\|_{(\alpha, \alpha)}$ and $\|y\|_{(\alpha, \beta)} = \|y\|_{(\alpha, \alpha)}$. By the proof of Proposition \ref{rp11}, $\|y_N\|_{(\alpha, \beta)}/\|x_N\|_{(\alpha, \beta)}\rightarrow\infty$ as $N\rightarrow\infty$. Therefore, $(e_i)_i$ is not $\mathcal{S}_{\alpha+1}$-unconditional.  
\end{proof}

\section{Further investigation}
It is natural for future work to investigate the following question: for $\alpha+2 \le \beta < \omega_1$, are there $(\alpha, \beta)$-quasi-greedy bases? If so, these bases would correspond to the empty circles in Figure \ref{qgtable}. 

\section{Acknowledgement}
We would like to thank the anonymous referee for carefully reading the paper.

\ \\
\end{document}